\documentclass[11pt,reqno]{amsart}
\usepackage[dvipsnames]{xcolor}
\usepackage{setspace}
\usepackage[hang,flushmargin,symbol*]{footmisc}
\usepackage{amsmath}
\usepackage{amsthm}
\usepackage{amssymb}
\usepackage{mathtools}
\usepackage{enumitem}
\usepackage{calc}
\usepackage{graphicx}
\usepackage{caption}
\usepackage[labelformat=simple,labelfont={}]{subcaption}
\usepackage{tikz}
\usepackage{tikz-cd}
\usetikzlibrary{decorations.markings}
\usetikzlibrary{arrows,shapes,positioning}
\usepackage{url}
\usepackage{array}
\usepackage{graphicx}
\usepackage{color}
\usepackage{mathrsfs}

\usepackage{verbatim}

\usepackage{subfiles}

\usepackage[colorinlistoftodos]{todonotes}

\usepackage{dashbox}

\theoremstyle{theorem}
\newtheorem{theorem}{Theorem}[section]
\newtheorem{lemma}[theorem]{Lemma}

\newtheorem{corollary}[theorem]{Corollary}
\newtheorem{proposition}[theorem]{Proposition}

\theoremstyle{definition}
\newtheorem{definition}[theorem]{Definition}
\newtheorem{example}[theorem]{Example}
\newtheorem{remark}[theorem]{Remark}

\usepackage{amsfonts}
\usepackage{amsmath}
\usepackage{amssymb}
\usepackage{mathrsfs}
\usepackage{tikz}
\usepackage{tikz-cd}
\usepackage[colorlinks=true]{hyperref}
\usepackage{ mathdots }
\usepackage{dutchcal}
\usepackage{verbatim}
\usepackage{amsmath,amsthm}
\usepackage{extarrows}

\usepackage{xparse}

\newcommand{\stquot}[1]{{[} #1 {]}}

\usepackage{multirow}

\usepackage[normalem]{ulem}
\renewcommand{\tilde}[1]{\ensuremath{\widetilde{#1}}}

\newcommand{\pt}{{\rm pt}}

\newcommand{\ZZ}{\mathbb{Z}}
\newcommand{\CC}{\mathbb{C}}

\newcommand{\QQ}{{\mathbb{Q}}}
\newcommand{\OO}{\mathcal{O}}

\newcommand{\PP}{\mathbb{P}}
\newcommand{\GG}{\mathbb{G}}

\newcommand{\Spec}{{\text{Spec}\:}}

\newcommand{\Hom}{\text{Hom}}

\newcommand{\pb}{{\arrow[dr, phantom, very near start, "\ulcorner"]}}

\newcommand{\Ms}[1][{g, n}]{\overline{M}_{#1}}
\newcommand{\Mp}[1][{g, n}]{\frak M_{#1}}

\newcommand{\Mx}[1][{\Xi}]{\tilde{M}_{#1}(S^d X)}

\newcommand{\Ks}[1][{\Xi}]{\cal K^*_{#1}(S^d X)}
\newcommand{\Kt}[1][{\Xi}]{\tilde{\cal K}_{#1}(S^d X)}
\newcommand{\Kb}[1][{\Xi}]{\bar{\cal K}_{#1}(S^d X)}

\newcommand{\Mxpt}[1][{\Xi}]{\tilde{M}_{#1}(BS_d)}

\newcommand{\Kspt}[1][{\Xi}]{\cal K^*_{#1}(BS_d)}

\newcommand{\Kbpt}[1][{\Xi}]{\bar{\cal K}_{#1}(BS_d)}

\newcommand{\longequals}{\xlongequal{\: \:}}

\newcommand{\action}{\:\rotatebox[origin=c]{-90}{$\circlearrowright$}\:}

\newcommand{\num}[1]{{\langle #1 \rangle}}

\renewcommand{\frak}[1]{\ensuremath{\mathfrak{#1}}}
\newcommand{\cal}[1]{\ensuremath{\mathcal{#1}}}

\def\Sym#1#2{[\operatorname{Sym}^{#1}#2]}

\renewcommand{\bar}[1]{\overline{#1}}

\newcommand{\inc}{{\ensuremath{\rm inc}}}

\newcommand{\KpBSd}{\tilde{\frak K}_{\Xi}(BS_d)}
\newcommand{\KpBSdstar}{\tilde{\frak K}_{\Xi^*}^*(BS_d)}
\newcommand{\KsSymX}{\Kt}
\newcommand{\KsSymXstar}{\Ks}

\newcommand{\KsSymptstar}{\tilde {\cal K}_{\Xi^*}^*(BS_d)}

\newcommand{\Aut}{\underline{\text{Aut}}}

\usepackage{stmaryrd}

\newcommand{\lcm}[1]{\mathrm{lcm}(#1)}

\definecolor{sebgreen1}{rgb}{0.019,0.317,0.149}
\definecolor{sebgreen2}{rgb}{0.784,0.952,0.780}

\newcommand{\Leo}[2][inline]{\todo[linecolor=purple,backgroundcolor=purple!25,bordercolor=purple,#1,shadow,author=Leo]{#2}} %Todo notes for LEO. 

\newcommand{\HH}{K_\circ}

\newcommand{\scr}[1]{{\ensuremath{\mathscr{#1}}}}

\newcommand{\instackbar}[1]{\overline{\rm I}\left(#1\right)}

\newcommand{\bra}[1]{{\left[{#1}\right]}}

\title{{Higher Genus Quantum $K$-theory}}
\author{You-Cheng Chou, Leo Herr, Y.-P. Lee}
\date{\today}

\begin{document}

\email{bensonchou@gate.sinica.edu.tw, herr@math.utah.edu, yplee@math.utah.edu, ypleemath@gate.sinica.edu.tw}

\address{Institute of Mathematics, Academia Sinica, Taipei 106319, Taiwan}

\address{Department of Mathematics, University of Leiden, Snellius Gebouw, Niels Bohrweg 1, 2333 CA Leiden, Netherlands}

\address{Department of Mathematics, University of Utah, Salt Lake City, Utah 84112-0090, U.S.A.}

\maketitle

\setcounter{section}{-1}

\begin{abstract}
We prove genus $g$ invariants in quantum $K$-theory are determined by genus zero invariants of a smooth stack in the spirit of K.~Costello's result in Gromov--Witten theory. 
%Quantum $K$-invariants of a stack are defined after Abramovich-Graber-Vistoli. 
%Combinatorial formulas for stabilizing cotangent line bundles $L_i$ are also obtained. We conclude with a simple example on elliptic curves. 
\end{abstract}

\section{Introduction}

Let $X$ be a smooth quasiprojective variety over $\mathbb{C}$.
Let $\Ms[g, R](X, \beta)$ be the space of genus--$g$, $R$--pointed stable maps to $X$ with degree $\beta$. The perfect obstruction theory on $\Ms[g, R](X, \beta)$ \cite{intrinsic} endows the moduli stack with a ``virtual structure sheaf'' $\OO^{vir}_{\Ms[g, R](X, \beta)}$ \cite{QK1}. 

Let $\alpha_i \in K^\circ(X)$ and $L_i$ be the universal cotangent line bundles. When the insertion
\[
 \Omega := \sum_I a_I \prod_{i=1}^R L_i^{k_i} \otimes ev_i^* \alpha_i
\]
has an action by 
%(a subgroup of) 
$S_R$, the permutation-equivariant pushforward
\begin{equation} \label{e:qkidef}
 \sum_j (-1)^j H^j \Big(\Ms[g, R](X, \beta), \OO^{vir} \otimes \Omega \Big)
\end{equation}
%\Leo{Does this make sense? $\OO^{vir}$ is not a sheaf, but a $K$ theoretic class}
is an element in the Grothendieck group of  
%(or its subgroup) 
$S_R$-representations with $\oplus$, i.e., virtual representations. We can also take a subgroup of $S_R$ instead. These are by definition the  \emph{permutation-equivariant quantum $K$-invariants}. 
%\Leo{only if the $\alpha_i$'s have some symmetry} Here $L_i$ are the cotangent line bundles at the $i$-th marked point. 
% \sum_{i=0} (-1)^i H^i \Big(\Ms[g, R](X, \beta), \OO^{vir} \otimes \prod_{i=1}^R L_i^{k_i} \otimes ev_i^* \alpha_i \Big)  

The main theorem of this paper is the following.

\begin{theorem} [Theorem~3.1]
%\Leo{We never explicitly work with permutation equivariant version. I think variations of our techniques would work, but it's wrong to say that's a theorem that we prove. }
%???? Permutation-equivariant ????
Genus $g$ quantum $K$-invariants on $X$ can be computed from permutation-equivariant genus zero quantum $K$-invariants on 
\[
 \Sym {g+1} X = \bra{X^{g+1}/S_{g+1}}.
\]
A similar statement holds for $X$ a smooth DM stack with projective coarse moduli space.
%\Leo{I have to think about this. It may complicate the proof of some pullbacks we're borrowing from \cite{mycostellogeneralization}.}
%\YP{Should we try to nail these down this time?}
\end{theorem}

We believe the \emph{permutation-equivariant} genus-zero Quantum $K$-theory of $X$ can also be computed from that of $\Sym {g+1} X$ by extending our methods.

%A few comments are in order. While we believe that the statement of the above theorem holds as ``\emph{permutation-equivariant quantum $K$-theory of higher genus can be effectively reduced to permutation quantum $K$-theory of genus zero}'', our current proof does not exactly imply that. A little bit of work is needed for this extension.
%\YP{Leo, is this what you meant?}
%\Leo{This is great!}

When the target, $X$ itself or $\Sym {g+1} X$, is a Deligne--Mumford stack, the definition of $\Ms[g, R](X, \beta)$ involves twisted/orbifold curves and twisted stable maps. The domain curves are families of pointed nodal curves with cyclic gerbe structures at the marked points and nodes, such that the gerbe structures at the nodes are \emph{balanced}. This means they are locally stack quotients of  the node $R[x,y]/(xy)$ by the antidiagonal action
\[\zeta.(x, y) \coloneqq (\zeta x, \zeta^{-1} y), \qquad \zeta \in \mu_r.\]
Note that we do \emph{not} require the $\mu_r$-gerbe structures at the marked points to be trivial in families as in \cite{costello}. \emph{Twisted stable maps} are \emph{representable} morphisms from twisted curves to the target with finite automorphisms. $\Ms[g, R](X, \beta)$ are the moduli stacks of (twisted) stable maps with the discrete data $g, R, \beta$.

Hence, the marked points are no longer literal ``points,'' but gerbes. Due to the nontrivial gerbes at the marked points, the evaluation maps have the natural codomain a partially rigidified inertia stack $\instackbar{X}$, instead of the inertia stack $IX$. %\coloneqq X \times_{X \times X} X$. % This issue can be resolved and the insertion classes for the quantum $K$-invariants come from a variant $\instackbar{X}$ of $IX$. 
This has been done in works by \cite{Chen_Ruan_2002, Chen_Ruan_2004} and \cite{abram-graber-vistoli} in the context of cohomology/Chow.

We only need one class pulled back from $\instackbar{-}$ as opposed to ordinary $K$ theory, which comes from $\Sym k X$ for some $k \leq g+1$ in Section \S \ref{ss:evalmaps}. We do not need the full $K$--theory of $\instackbar{\Sym {g+1} X}$.

%\YP{Notation-wise, do we need to put the subscript $\bar{\mu}$? I don't think we fixed $\mu$?}
%\Leo{Notation is changed}
%\Leo{I think we use $\instackbar{X}$ instead, a relative.}

The appearance of permutation-equivariant $K$-theory is quite natural, not simply a ``technical clutch''. In cohomological Gromov--Witten theory, we often rely on the fact that the substacks (``strata'') appearing in ``common operations'' (e.g., fixed point loci in torus localizations or the components of inertia stacks of the moduli) are \emph{variants} of known quantities in the sense of induction. These variants can be identified with the actual known quantities in Gromov--Witten theory by simple modifications. For example, for the purpose of computing Gromov--Witten invariants, 
\[
 \int_{[\overline{M} / S_n ]} ... = \frac{1}{n!} \int_{[\overline{M}]} \pi^*(...).
\]
These equalities are no longer true in Quantum $K$ theory. In fact,
\[
 \chi ([\overline{M} / S_n ], ...) = \chi_{S_n}(\overline{M}, \pi^*(...) )^{S_n} .
\]
We note that $K$-theory on $[\overline{M} / S_n ]$ can be identified with the $S_n$-equivariant $K$-theory on $\overline{M}$, and $\chi_{S_n}(-)^{S_n}$ is the pushforward in the $S_n$-equivariant theory, i.e., the $S_n$-invariant part $(-)^{S_n}$ of alternating sum of sheaf cohomologies viewed as $S_n$ representations $\chi_{S_n}(-)$. This necessitates permutation-equivariant quantum $K$-theory.

Quantum $K$-theory has already been defined for stacks in e.g., \cite{kthytoricstacks1} and \cite[\S 2.4]{kthytoricstacks}. A comparison of the quantum $K$-theories with trivial and nontrivial gerbes at marked points can be found in \cite[Remark 2.8]{kthytoricstacks}. See also \cite[\S 4.4, 4.5]{AGVproceedingsfirstversion} in the cohomological context. We allow nontrivial gerbes and recall the basic definitions in Section \S \ref{ss:gwkthyforstacks}. 

%\YP{1. Do we need $K(IY)$ or only $K(Y)$ in our insertions? Here $Y =\Sym {g+1} X$ or $X$ etc..}
%\Leo{Good question. I added something about it above when we introduced the inertia stack. We need slightly more than just $K(\Sym d X)$, but only slightly, and I said as much.}

%\Leo{My best interpretation of what you were saying is that there is a map from each component $Y_c \subseteq IY$ to $Y$ itself $i : Y_c \to Y$. If $Y = V/G$, then $Y_c$ looks like a closed subset of $V$ mod a subgroup of $G$. Since we're factoring through a particular component $Y_c$, maybe $i^* i_* \alpha = \alpha$ and we can pull back this class to $\Ms(X)$ to integrate instead. This may or may not be true, but there's two subtleties: 1) The map $\Ms(X) \to Y_c \to Y$ is no longer really the evaluation map at the marked point in the usual sense, just some map 2) We're not using inertia stacks. We're using something a bit more complicated, $\instackbar{X}$. The story is not as simple for this. }

%\YP{Not what I meant. Not to over-generalize, do we really need to use any class which can not be pulled back from $K(Y)$? If so, where?}

%\YP{2. Further discussions on the codomain of evaluation maps, cf.\ Sections 4.4 and 4.5 in AGV proceeding paper }
%\Leo{I added a citation. }

Quantum $K$-invariants are roughly Gromov--Witten invariants computed in $K$-theory instead of cohomology or Chow groups. The idea of computing genus-$g$ Gromov--Witten invariants of any smooth projective variety $X$ in terms of genus 0 quantum $K$-invariants of quotient stack $\Sym {g+1} X$ goes back to M.~Kontsevich and was independently obtained in K.~Costello's thesis \cite{costello}. This paper can be considered as a $K$-theoretic version of this circle of ideas.

%There are many ways that this reduction from higher genus to genus zero might prove useful in the future. 
The calculation of genus zero Quantum $K$-theoretic invariants is simpler and self-contained, while the higher genus invariants necessarily involve invariants of lower genus. Genus-zero quantum $K$-theory is much better understood, with additional \emph{finite difference} structure in addition to the usual $D$-module structure.

Quantum $K$-theory has 
%become a more visible topic in recent years through its 
connections with modern enumerative geometry, integrable systems, representation theory, geometric combinatorics and theoretical physics. Its influence on theoretical physics is largely its relation to $3$ dimensional topological field theory. See the pioneering works of N.~Nekrasov, H.~Jockers, P.~Mayr etc. (\cite{jockerspeterquantumkphysics}, \cite{jockerspeterquantumkphysics2}, and references therein.) For its connection to representation theory, see \cite{Okounkov_lectures} and references therein. At the very onset of the quantum $K$-theory, it was intimately connected to integrable systems. See, for example, \cite{Givental_Lee}. It has also inspired much progress in geometric combinatorics through works like \cite{Buch_2011, Buch_2013, Buch_2020} of A.~Buch, P.~Chaput, C.~Li, L.~Mihalcea, N.~Perrin and many others. Most of these works are in \emph{genus zero}. We hope that our algorithm will prove useful in the further development of higher genus quantum $K$-theory.

%\Leo{Add motivation. Why do we want to do this?}

We work exclusively with schemes, stacks, etc. locally of finite type over the complex numbers $\CC$. In particular, they are locally noetherian.

\subsection*{Acknowledgments}
The authors wish to thank 
Sarah Arpin,
Peter Bakic, 
Sebastian Bozlee,
Renzo Cavalieri, 
Kevin Costello,
Dan Edidin,
Adeel Khan, 
Rufus Lawrence, 
Sam Molcho, 
Hsian-Hua Tseng,
Jonathan Wise, 
and the Mathoverflow community for their help and support.
This project originated in the second author's work with J.~Wise \cite{mycostellogeneralization}. %on K.~Costello's pushforward theorem in Chow groups  %Costello's original pushforward theorem was widely used but wrong as stated, so they corrected it and fixed twenty sample applications of the theorem. We then proved Costello's pushforward theorem in $K$-theory, and this is our first sample application. 
%\YP{If Leo agrees, I prefer not to comment on Costello's pushforward formula. It should be left to the Herr--Wise paper, no? Furthermore, didn't we already use it in the product paper (to a lesser degree)?}

The first and the third authors wish to acknowledge the partial supports from Academia Sinica, National Science and Technology Council, Simons Foundation and the University of Utah. The second author thanks Leiden University, University of Utah and the National Science Foundation for its support through RTG Grant \#1840190. 
%\Leo{I have to put this}

\section{Higher genus quantum $K$-invariants}

Let $C'$ be a general genus $g$ smooth curve with a general divisor $B$ of degree $d=g+1$. There is exactly one ramified cover $f : C' \to \PP^1$ of degree $d$ with ramification divisor $B = f^* \infty$ over infinity by Riemann-Roch \cite[Lemma 6.0.1]{costello}, \cite[Theorem 3.12]{mycostellogeneralization}. The following facts come from K.~Costello \cite{costello}.
\begin{itemize}
    \item This entails a birational map between moduli spaces (Lemma \ref{lem:propbirational}). 
    \item By adding stack structure $\tilde{\cdot}$, we can make $f : \tilde C' \to \tilde{\PP^1}$ a finite \'etale cover. This is pulled back along a map $\tilde{\PP^1} \to BS_d$ from a stacky genus zero curve to the moduli space $BS_d$ of finite \'etale degree-$d$ covers. 
\end{itemize}

We can similarly interpolate between genus-$g$ and genus-zero maps to a fixed smooth, quasiprojective target $X$.

\begin{remark}

%\marginnote{I changed $T' \to T$ to $T' \to X$. $\num{d}$ not yet defined?}
Write $\num{d} = \{1, 2, \dots, d\}$ for the ordered set of $d$ elements. An $S_d$--torsor $P \to X$ is equivalent to the data of a finite étale degree $d$ cover 
\[T' := (P \times \num{d})/S_d \to X. \]
%induced via the antidiagonal quotient
%These are pulled back from the universal $d$-sheeted cover $E = (\pt \times_{BS_d} \num{d})/S_d$. The stack $BS_d$ can be thought of as the moduli stack of finite étale degree $d$ covers. 
The universal $d$-sheeted cover 
\[ (\pt \times \num{d})/S_d \to BS_d \] 
%\marginnote{Do we really need this non-canonical identification?}
%\Leo{Yes! It helps me. }
can be non-canonically identified with the map $BS_{d-1} \to BS_d$ induced by any of the $d$ inclusions $S_{d-1} \subseteq S_d$.
\end{remark}

Consider a twisted, representable stable map $\tilde C' \to X$ together with a finite \'etale degree-$d$ cover $\tilde C' \to \tilde C$ of a stacky curve $\tilde C$ of genus zero. To promote $\tilde C'$ to a marked curve, we need only order the fibers of the marked points of $\tilde C$.

Our data is pulled back from the universal finite \'etale degree-$d$ cover mapping to $X$:
\[
\begin{tikzcd}
%\tilde C' \ar[r] \ar[d] \pb      &X \times \Sym {d-1} X \ar[d]      \\
\tilde C' \ar[r] \ar[d] \pb      &{[ (X^d \times \num{d})/S_d ]} \ar[d]      \\
\tilde C \ar[r]        &\Sym d X = [X^d /S_d ],
\end{tikzcd}
\]
and the whole diagram has finitely many automorphisms over the right arrow if and only if the map $\tilde C \to \Sym d X$ is stable.

\begin{definition}

The stack $\tilde {\cal K}_{0, n}(\Sym d X)$ parameterizes families $\tilde C \to S$ of twisted curves of genus zero with $n$ marked points and a representable map $\tilde C \to \Sym d X$ together with an ordering of the fibers over the marked points. The marked points of $\tilde C$ may be nontrivial gerbes over $S$. 
    
\end{definition}

%\marginnote{not precise. either write down details or quote Costello with specific page numbers.}
The stack $\tilde {\cal K}_{0, n}(\Sym d X)$ equivalently parametrizes families of ramified $d$-sheeted covers $C' \to C$ together with maps $C' \to X$ that have finitely many automorphisms. All ramification points are marked and the fibers above the marked points of $C$ are all the marked points of $C'$. By ``ordering of the fibers,'' we mean that the fibers of $C'\to C$ over each marked point of $C$ must be ordered, a torsor for a product of symmetric groups. We later consider variants where less of the marked points of $C'\to C$ are ordered; see Figure \ref{fig:zoostacks}.

Our twisted/stacky stable maps and curves are different from \cite{costello}. For families of curves $\tilde C \to S$ over a base scheme $S$, the $i$th marked point of $\tilde C$ may be a \textit{nontrivial} $\mu_{r_i}$ gerbe, for $r_i \in \ZZ_{\geq 1}$. We fix the orders $r_i$ later.

We want to apply the $K$-theoretic version of Costello's pushforward formula \cite[Theorem 2.7]{logquantumkproductchouherrlee} to a square from \cite[\S 3.2]{mycostellogeneralization} introduced in \S \ref{ss:costellosquare}:

\begin{equation}\label{eqn:costellosquare}
\begin{tikzcd}
\KsSymX \ar[r, "q"] \ar[d, "\pi'", swap] \ar[dr, phantom, very near start, "\ulcorner"]     &\Ms[g, R](X) \ar[d, "\pi"]     \\
\KpBSd \ar[r, "p", swap]      &\Mp[g, R].
\end{tikzcd}\end{equation}
% \begin{comment}
% \begin{equation}\label{eqn:costellosquare}
% \begin{tikzcd}
% \MsSymG \ar[r] \ar[d] \pb     &\Ms[g, dn-A](X) \ar[d]        \\
% \MprelSdG \ar[r, "p'"]       &\Mp[g, dn-A].
% \end{tikzcd}    
% \end{equation}
% \end{comment}

The stacks $\Ms[g, R](X)$ of stable maps and $\Mp[g, R]$ of prestable curves are standard. We do not fix a curve class $\beta$, so this space is a disjoint union over choices of $\beta$. 

We denote $\KsSymX \subseteq \tilde {\cal K}_{0, n}(\Sym d X)$ a substack with appropriate discrete invariants fixed in \S \ref{ss:specxi}. The stack $\KpBSd$, denoted $\tilde{\frak M}_{0, n}(BS_d)$ in \cite[before Lemma 3.6]{mycostellogeneralization}, is approximately the stack of prestable maps $\tilde C \to BS_d$ from the genus zero twisted base curves parameterized in $\KsSymX$.

The obstruction theory for $\pi'$ is pulled back from $\pi$. The problem is that $p$ is of degree
\[e = k!(g!)^{\#J}(g!)^k,\]
while the $K$-theoretic virtual pushforward formula so far only applies to birational maps. We decompose $p$ as a finite \'etale torsor of degree $e$ composed with a birational map to which the pushforward formula applies.

\subsection{Costello's Square \eqref{eqn:costellosquare}} \label{ss:costellosquare}

We describe \eqref{eqn:costellosquare}. Write $\num d = \{1, 2, \dots, d\}$. A subset $A \subseteq \num{\ell}$ will be fixed later; the symbols $\Ms[g, R](X)$, $\Mp[g, R]$ refer to moduli stacks of ordinary stable maps and prestable curves with $R = \ell - \# A$ marked points. We do not fix the curve class $\beta$ for simplicity. We assume $R \geq 1$. 

The substack $\Kt \subseteq \tilde {\cal K}_{0, n}(\Sym d X)$ parametrizes stable maps of genus-zero curves to $\Sym d X$, identified with triples $C \leftarrow C' \to X$ above. The $\Xi$ refers to fixed discrete invariants fixed in \ref{ss:specxi}: ramification profiles of $C' \to C$, the numbers $n$ and $\ell$ of marked points for $C$ and $C'$, the genus of $C'$, and the degree $d$ of $C' \to C$. The number $\ell \leq dn$ is the sum of the degrees of the fibers in the ramification profiles. These invariants satisfy Riemann-Hurwitz to ensure the space is nonempty:
\begin{equation}\label{eqn:riemannhurwitz}
2g-2 = -2 d + \sum_{P \in C'} (e_P -1).    
\end{equation}

The functor $q$ forgets the marked points $A \subseteq \num{\ell}$ of $C'$ and then takes the \textit{stabilization} $\overline{C}' \to X$ of the resulting map $C' \to X$.

The map $\pi'$ forgets the stable map to $X$. To make the diagram commute, $\pi'$ must remember the stabilization $\overline{C}'$ of $C' \to X$. Define the stack $\KpBSd$ of triples $C \leftarrow C' \to D$, where $C' \to C$ is a ramified cover of type $\Xi$ and $C' \to D$ a partial stabilization after forgetting $A \subseteq \num{\ell}$. The map $p$ sends this triple to $D$. The square \eqref{eqn:costellosquare} is cartesian and $p$ is proper by Lemma 3.9 and Corollary 3.7 of \cite{mycostellogeneralization}, respectively.

\begin{remark}

The degree $e = k! (g!)^{\# J} (g!)^k$ differs from both \cite[Theorem 3.12]{mycostellogeneralization} and \cite[Lemma 6.0.1]{costello}. Our use of nontrivial gerbes instead of trivialized gerbes accounts for the difference from \cite{mycostellogeneralization}. Costello's version is reconciled in Remark 3.15 of loc. cit. Our degree can be computed using the proof of Theorem 3.12 of loc. cit. or by taking into account the degrees of the universal gerbes. 

The degree $e$ is the order of a group $\Gamma = (S_g)^J \times S_g \wr S_k$ that reorders marked points of $C' \to C$ discussed in \S \ref{ss:Kstar}. 

The stabilization $\overline{C}' \to X$ was omitted in \cite{costello}, leading to non-proper moduli stacks or noncommutative diagrams. This could be rectified using his technology of weighted graphs instead of our partial stabilizations. 

\end{remark}

\subsection{Specifying $\Xi$}\label{ss:specxi}

\begin{figure}
    \centering
\begin{tikzpicture}
\node at (-1, .75){$C'$};
\draw[->] (-1, .45) to (-1, -1.25);
\node at (-1, -1.5){$C$};
\draw[-] (0, 0) to (3, 0);
\draw[-] (0, 0.5) to (3, 0.5);
\draw[-] (0, 1) to (3, 1);
\draw[-] (0, 1.5) to (3, 1.5);
\draw (3, 1) arc (-90:90:.25);
\draw (3.5, 1.5) arc (90:270:.25);
\draw[-] (3.5, 1.5) to (6, 1.5);
\draw (4, .5) arc (-90:90:.25);
\draw (4.5, 1) arc (90:270:.25);
\draw (5, 0) arc (-90:90:.25);
\draw (5.5, 0.5) arc (90:270:.25);
\draw[-] (3.5, 1) to (4, 1);
\draw[-] (4.5, 1) to (7, 1);
\draw[-] (3, .5) to (4, .5);
\draw[-] (4.5, .5) to (5, .5);
\draw[-] (5.5, .5) to (6, .5);
\draw[-] (3, 0) to (5, 0);
\draw[-] (5.5, 0) to (7, 0);
\draw (6, 0.5) arc (-90:90:.5);
\draw (7, 1.5) arc (90:270:.5);
\draw[-] (0, -1.5) to (7, -1.5);
\node[below] at (1.5, -2){$J$};
%nodes on base for J
\fill[white] (0.25, -1.5) circle (.07cm);
\fill[white] (1.25, -1.5) circle (.07cm);
\fill[white] (1.95, -1.5) circle (.07cm);
\draw (0.25, -1.5) circle (.07cm);
\draw (1.25, -1.5) circle (.07cm);
\draw (1.95, -1.5) circle (.07cm);
%nodes on source for J
\fill (0.25, 0) circle (.05cm);
\fill (1.25, 0) circle (.05cm);
\fill (1.95, 0) circle (.05cm);
\fill (0.25, 0.5) circle (.05cm);
\fill (1.25, 0.5) circle (.05cm);
\fill[white] (1.95, 0.5) circle (.07cm);
\draw (1.95, 0.5) circle (.07cm);
\fill (0.25, 1) circle (.05cm);
\fill[white] (1.25, 1) circle (.07cm);
\draw (1.25, 1) circle (.07cm);
\fill (1.95, 1) circle (.05cm);
\fill[white] (0.25, 1.5) circle (.07cm);
\draw (0.25, 1.5) circle (.07cm);
\fill (1.25, 1.5) circle (.05cm);
\fill (1.95, 1.5) circle (.05cm);
%nodes on source for k
\fill (3.25, 0) circle (.05cm);
\fill (3.25, 0.5) circle (.05cm);
\fill (3.25, 1.25) circle (.05cm);
\fill (3.25, 0) circle (.05cm);
\fill (3.25, 0.5) circle (.05cm);
\fill (3.25, 1.25) circle (.05cm);
\fill (4.25, 0) circle (.05cm);
\fill (4.25, 0.75) circle (.05cm);
\fill (4.25, 1.5) circle (.05cm);
\fill (5.25, 0.25) circle (.05cm);
\fill (5.25, 1) circle (.05cm);
\fill (5.25, 1.5) circle (.05cm);
%nodes on base for k
\fill (3.25, -1.5) circle (.07cm);
\fill (4.25, -1.5) circle (.07cm);
\fill (5.25, -1.5) circle (.07cm);
\node[below] at (4.5, -2){$k$};
%nodes on source for I/\infty
\fill[white] (6.5, 1) circle (.07cm);
\draw (6.5, 1) circle (.07cm);
\fill[white] (6.5, 0) circle (.07cm);
\draw (6.5, 0) circle (.07cm);
%node on base for \infty
\fill[white] (6.5, -1.5) circle (.07cm);
\draw (6.5, -1.5) circle (.07cm);
\node[below] at (6.5, -2){$\infty$};
\end{tikzpicture}
    \caption{\cite[Figure 2]{mycostellogeneralization}. A cover in $\Xi$, $g=3, d=4$. Marked points are colored black if forgotten $A \subseteq \num \ell$ and white if remembered under the map to $\Mp[g, R]$. The space $\KpBSdstar$ forgets the ordering on the black marked points. }
    \label{fig:excostellodiscretedata}

\end{figure}

We unpack our discrete data:
\[\Xi = 
\left\{\begin{tikzcd}[row sep=small]
g(C') = g, \,g(C) = 0, \,d = g+1,        \\
\num{\ell} \to \num{n} \text{ is } \num{k} \times \num{d} \overset{pr_1}{\mapsto} \num{k},\,\, J \times \num{d} \overset{pr_1}{\mapsto} J,\,\, I \mapsto \infty, \\
\forall j \in J,\,\, r_j = 1, \forall j \in \num{k}\,\, r_j = 2, \gamma : I \to \ZZ_{\geq 1}, r_{\infty} = \lcm{\gamma(i)},       \\
\bigsqcup_J B\mu_1 = * \to BS_d,\,\, \bigsqcup_{i \in \num{k}} B\mu_2 \overset{\phi}{\to} BS_d
\end{tikzcd}\right\}\]
See Figure \ref{fig:excostellodiscretedata} for an example.

Ramification profiles are specified by an action of $\mu_r$ on an unordered set of size $d$. Take a small loop around $p \in C$, and its lifts to $C'$ identify which of the $d$ sheets come together over $p$. Encode this action in a map $B\mu_r \to BS_d$ up to isomorphism.

\begin{remark}\label{rmk:mapsofgerbes}

The category of maps $BG \to BH$ has 
\begin{itemize}
    \item Objects: homomorphisms $f : G \to H$
    \item Morphisms $f_1 \to f_2$: elements $h \in H$ which conjugate one morphism to another $f_1 = h f_2 h^{-1}$. They are all isomorphisms. 
\end{itemize}
Objects of the category $\Hom(B\mu_r, BS_d)$ can be identified with actions $\mu_r \action \num{d}$. Isomorphisms between two such actions are relabelings of the set $\num{d}$ of $d$ elements. So an isomorphism class of functors $B\mu_r \to BS_d$ is an action of $\mu_r$ on an unlabeled set with $d$ elements. 

The action $\mu_r \action \num{d}$ contains the information of a ramification point $C' \to C$ of a map of curves. The stacky quotient $\bra{\num{d}/\mu_r}$ is the fiber of $\tilde C'\to \tilde C$ over the point $B\mu_r \in C$. To extract the set-theoretic fiber, we take the coarse moduli space $\num{d}/\mu_r$. This gives an unlabeled set with some number of elements between 1 and $d$. In families, $B\mu_r$ is allowed to be a nontrivial gerbe. 

\end{remark}

A point $p \in C$ is \emph{simply ramified} if its fiber consists of $d-1$ points, where exactly two of the $d$ sheets come together and the other points in the fiber are unramified. This corresponds to the action $\mu_r \action \num{d}$ with one 2-cycle and the rest of the points fixed, up to reordering $\num{d}$.

Let $k \geq 0$ be an integer, $g = g(C')$ be the genus of $C'$, and fix the degree $d = g + 1$. Divide the $n$ marked points of $C$ into three sets:
\begin{itemize}
    \item[$\{\infty\}$]: Write $I \subseteq C'$ for the fiber over this point $\infty \in C$. This point has ramification described by a map $B\mu_{r_\infty} \to BS_d$ or function $\gamma : I \to \ZZ_{\geq 1}$. That is, $\gamma(i)$ is the size of the stabilizer of $i$ in the corresponding action $\mu_{r_\infty} \action \num{d}$. 
    
    \item[$J$]: These points $J \subseteq C$ have no ramification. 
    
    \item[$\num{k}$]: These points have simple ramification. 
\end{itemize}

\begin{comment}
All the marked points of $C$ except possibly one $\infty \in C$ will either be simply ramified or unramified. Let $I \subseteq C'$ be the fiber over $\infty \in C$ and let $\gamma : I \to \ZZ_{\geq 1}$ encode the stack structure/ramification at each point. That is, $\gamma(i)$ is the size of the stabilizer of $i$ in the above action $\mu_r \action \num{d}$. 
\end{comment}

This gives partitions:
\begin{align*}
    \num{n} &= J \sqcup \num{k} \sqcup \{\infty\},      \\
    \num{\ell} &= J \times \num{g+1} \sqcup \num{kg} \sqcup I.
\end{align*}
The map $\num{\ell} \to \num{n}$ on marked points is compatible with these partitions.

The $j$th marked point of $C$ is a $\mu_{r_j}$-gerbe, where $r_j = 1$ at unramified points $j \in J$, $r_j = 2$ for the $k$ simple ramification points and $\infty$ has $r_\infty = \lcm{\gamma(i)}$ the least common multiple of the ramification function $\gamma$ on $I$. These data are subject to a constraint easier seen with trivialized gerbes: the sum
\[B\ZZ \to BS_d\]
of the classifying space maps from all the composites $\ZZ \to \mu_r \to S_d$ be zero, lest the space be empty. This corresponds to the presentation of $\pi_1(\PP^1 \setminus \num{n})$ via generators whose product is trivial. 

Let $A \subseteq \num{\ell}$ consist of all of $\num{kg}$, none of $I$, and a subset of $J \times \num{g+1}$ such that $J \times \num {g+1} \setminus A \to J$ is a bijection. Note that the set $\num{\ell} \setminus A = I \sqcup J$ has $R$ elements.

Take  $k = \#I + 3g -1$ so that all the dimensions agree \cite[Theorem 3.12]{mycostellogeneralization}:
\[\dim \KpBSd = \dim \Mp[g, R].\]

We can now prove our main equality between virtual fundamental classes. We first recall their definition in $K$-theory.

\subsection{$K$-theory}

Let $Y$ be a finite type noetherian algebraic stack. The $K$-theory of $Y$ is the $K$-theory of a category of lis-ét sheaves on $Y$, for which there are two main options:
\begin{itemize}
    \item $K_\circ(X)$: coherent sheaves, otherwise known as $G$-theory $G(Y)$. 
    \item $K^\circ(Y)$: locally free sheaves of finite rank. 
\end{itemize}
We work with $\QQ$-coefficients, tensoring these groups up to $\QQ$-vector spaces
\[K_\circ(X) = K_\circ(X) \otimes \QQ, \qquad K^\circ(Y) = K^\circ(Y) \otimes \QQ. \]

These groups are generated by classes $[F]$ of coherent/locally free lis-ét sheaves $F$ on $Y$, modulo relations $[F'] + [F''] = [F]$ for each exact sequence
\[0 \to F' \to F \to F'' \to 0.\]
See \cite[\S 1]{logquantumkproductchouherrlee} for discussion. 

The groups $K^\circ, K_\circ$ coincide on $X$ whenever every coherent sheaf $F$ admits a finite resolution by locally free sheaves. Under certain hypotheses on $X$, this is equivalent to $X$ being a quotient stack \cite[Remark 2.15]{brauergroupsandquotientstacks}.

Let $f : X \to Y$ be a map between finite type noetherian algebraic stacks. Pullback and pushforward of sheaves sometimes induce maps on $K^\circ$ and $K_\circ$. 
\begin{itemize}
    \item $K^\circ$: pullback $f^*$ always exists and pushforward $f_*$ makes sense when $X \to Y$ is finite étale.
    \item $K_\circ$: pullback $f^*$ exists when $f$ is flat. Armed with a perfect obstruction theory, we can also define a pullback $f^!$ even if $f$ is not flat. 
    
    If $f$ is proper and of DM type, define the pushforward $f_*$ on $K_\circ$ theory as the alternating sum
    \begin{equation}\label{eqn:pfwddmtypekthy}
    f_* F \coloneqq \sum_i (-1)^i R^i f_* F.
    \end{equation}
    We must check that this sum is finite. 
\end{itemize}

When the map $f : X \to Y$ is clear from context, we write $\beta|_X = f^* \beta$ for classes $\beta \in K^\circ(Y)$ or $\beta \in K_\circ(Y)$ without risk of confusion.

\begin{lemma}

Let $p : X \to Y$ be a proper, DM type morphism between finite type noetherian algebraic stacks. The pushforward 
\[p_* : K_\circ(X) \to K_\circ(Y).\] 
of \eqref{eqn:pfwddmtypekthy} is well-defined on $K_\circ$ theory.

\end{lemma}

\begin{proof}

We argue that the sum \eqref{eqn:pfwddmtypekthy} is finite. Write $\bar X$ for the relative coarse moduli space of the map $p$ and $N$ for a number larger than the dimensions of the fibers of $\bar X \to Y$. This is possible using quasicompactness of $Y$. 

We claim $R^i p_* F$ vanishes for $i > N$ for any coherent sheaf $F$. The claim is étale local in $Y$ and so is the relative coarse moduli space $\bar X$, so we can assume $Y$ is an affine scheme. The map $t : X \to \bar X$ is finite flat, so pushforward is exact $Rt_* = t_*$. This reduces to the representable case $\bar X \to Y$.

\end{proof}

Pushforward from a proper DM stack to a point is denoted $\chi$.

\begin{example}

If the morphism $p : X \to Y$ is not of DM type, the pushforward need not be well defined. Take $B\GG_m \to \pt$. The cohomology of $B\GG_m$ is freely generated as a ring by the first Chern class of the universal line bundle and does not vanish in any degree. 
%\Leo{Do we have an example of one that's proper but not DM type? Maybe $BE$ for an elliptic curve or $BPGL$ if PGL is projective?}

\end{example}

\begin{example}\label{ex:pfwdtrivialgerbe}

Let $G$ be a finite group. Sheaf pushforward along $p : BG \to \pt$ sends a complex $G$-representation $V$ to the invariant subspace $V^G$. The pushforward is then the alternating sum of group cohomology
\[\chi(V) = \sum (-1)^i [H^i(G, V)].\]
Because we work over $\CC$, the structure sheaf of $\pt$ is $\OO_{pt} = \CC$. Likewise $G$-representations $V$ on $BG = \bra{\Spec \CC/G}$ are complex representations, and the order of the group $\#G$ is invertible in $V$. The group cohomology therefore vanishes
\[H^i(G, V) = 0 \qquad i \neq 0.\]
The alternating sum is just the first term
\[\chi(V) = [V^G].\]

\end{example}

The projection formula holds in both $K^\circ$ and $K_\circ$, where defined
\begin{equation}\label{eqn:projectionfmla}
f_*(\alpha \otimes f^* \beta) = f_*\alpha \otimes \beta.    
\end{equation}
This results from the formula on the level of sheaves \cite[08EU]{sta}.

The main $K$-theory classes we are interested in are the fundamental class $[\OO_X] \in K_\circ(X)$ and the \textit{virtual fundamental class} (a.k.a. \textit{virtual structure sheaf}) \cite[\S 2.3]{QK1}, \cite[Definition 2.2]{fengquktheory}, \cite[Definition 1.2]{logquantumkproductchouherrlee}. Consider a map $f : X \to M$ from a DM stack $X$ to a smooth stack $M$ endowed with a perfect obstruction theory $C_{X/M} \subseteq E$. The virtual fundamental class $[\OO_X^{vir}]$ is the image of the structure sheaf of the normal cone $[\OO_{C_{X/M}}]$ under the isomorphism \cite[Remark 1.6]{logquantumkproductchouherrlee}
\[[\OO_X^{vir}] = \sigma^*[\OO_{C_{X/M}}] \qquad  \sigma^* : K_\circ(E) \simeq K_\circ(X).\]

\begin{example}\label{ex:pfwdvfctrivialgerbes}

Let $\pi : Y = BG \times X \to X$ be a trivial gerbe for a finite group $G$. Suppose $X$ has a perfect obstruction theory over some $M$ and $Y$ is given the induced perfect obstruction theory. Then the virtual fundamental class pulls back $\pi^* [\OO^{vir}_X] = [\OO_Y^{vir}]$. 

Example \ref{ex:pfwdtrivialgerbe} describes $\pi_*$ as taking $G$-invariants of a representation. Then $\pi_* \pi^* [\OO_X] = [\OO_X]$. Using the projection formula, this implies the virtual fundamental class also pushes forward
\[\pi_*[\OO_Y^{vir}] = \pi_*(\pi^*[\OO_X^{vir}] \otimes [\OO_Y]) = [\OO_X^{vir}] \otimes \pi_* \pi^*[\OO_X] = [\OO_X^{vir}].\]

\end{example}

\begin{proposition}\label{prop:pfwdgerbestrsheaf}

Let $\pi : \cal G \to X$ be a gerbe banded by a finite group $G$. The base $X$ is a scheme or algebraic stack which we emphasize lies over $\CC$. Then the structure sheaf pushes forward to the structure sheaf, both as sheaves and in $K$-theory:
\[R\pi_* \OO_{\cal G} = \pi_* \OO_{\cal G} = \OO_X, \qquad \pi_*[\OO_{\cal G}] = [\OO_X] \in \HH(X).\]

The same holds for virtual fundamental classes if $\cal G$ is given the induced perfect obstruction theory from $X$
\[\pi_* [\cal G]^{vir} = [X]^{vir} \qquad \in \HH(X). \]

\end{proposition}

\begin{proof}

The statement on sheaves implies that on $K$-theoretic classes and is local in $X$. We can then assume $\cal G$ is trivial, fitting in a pullback square
\[
\begin{tikzcd}
\cal G \ar[r] \ar[d] \pb     &BG \ar[d]         \\
X \ar[r]       &\pt.
\end{tikzcd}
\]
Example \ref{ex:pfwdtrivialgerbe} covers the case of $BG \to \pt$, and the general case results from cohomology and base change applied to this square. 

The statement on virtual fundamental classes results from the first as in Example \ref{ex:pfwdvfctrivialgerbes}.

\end{proof}

\begin{remark}

The proof of Proposition \ref{prop:pfwdgerbestrsheaf} does not work for schemes over $\ZZ$. The groups $H^i(G, V)$ for $i \neq 0$ are torsion, and so are the sheaves $R^i\pi_* V$ for any coherent sheaf on $\cal G$. But this does not mean they vanish in $\HH \otimes \QQ$. Tensoring $-\otimes \QQ$ kills $K$-theoretic classes that are torsion in the group law on $K$-theory, not the classes of sheaves that themselves are torsion. 

The trivial gerbe $\bra{\Spec \ZZ/G} \to \Spec \ZZ$ does satisfy Example \ref{ex:pfwdtrivialgerbe} and Proposition \ref{prop:pfwdgerbestrsheaf}, because the classes of torsion groups vanish in the $K$-theory of the integers. But this statement does not localize. 

\end{remark}

We need two related theorems on the behavior of (virtual) fundamental classes under pushforward. These extend Hironaka's theorem and Costello's theorem, respectively. 

\begin{theorem}[{Hironaka's pushforward theorem \cite[Proposition 2.3]{logquantumkproductchouherrlee}}]\label{thm:ourhironakathm}\phantom{a}
Let $p : X \to Y$ be a proper birational map of smooth DM stacks. The pushforward of the fundamental class of $X$ is that of $Y$ in $K$-theory
\[p_*[\OO_X] = [\OO_Y] \qquad \in K_\circ(Y). \]
\end{theorem}

Our Costello-type pushforward theorem was originally in the more general context of log geometry. We remove log structures in our citation for simplicity. 

\begin{theorem}[{Costello's pushforward theorem \cite[Theorem 2.7]{logquantumkproductchouherrlee}}]\label{thm:ourcostellothm}\phantom{a}
Consider a pullback square of algebraic stacks
\[\begin{tikzcd}
X \ar[r, "p"] \ar[d] \pb       &Y \ar[d]      \\
M \ar[r, "q"]       &N
\end{tikzcd}\]
with $X, Y$ DM stacks and $M, N$ smooth. Suppose $Y \to N$ is equipped with a perfect obstruction theory and $X \to M$ is given the pullback perfect obstruction theory. If the map $q$ is proper birational, the pushforward of the virtual class of $X$ is that of $Y$ 
\[p_*[\OO_X^{vir}] = [\OO_Y^{vir}] \qquad \in K_\circ(Y).\]
\end{theorem}

To use these theorems, it is important that the relevant maps are proper and birational. Birational maps $f : X \to Y$ of stacks must have an open dense subset of each $X$ and $Y$ that are isomorphic. 

\begin{remark}\label{rmk:puredegneqbirational}

For stacks, \textit{pure degree one} \cite[Definition 2.3]{mycostellogeneralization} and birational are not the same. The map $p : B\ZZ/2 \sqcup B \ZZ/2 \to \pt$ is pure degree one but not birational. The pushforward of the fundamental class is not the fundamental class:
\[p_*[\OO_{B\ZZ/2 \sqcup B \ZZ/2}] = 2 \cdot [\OO_\pt] \qquad \in K_\circ(\pt).\]

If a morphism of schemes is of pure degree one, it is birational. More generally, if $X \to Y$ is a morphism of stacks of pure degree one inducing a \textit{representable} morphism $U \to V$ on open dense substacks $U \subseteq X, V \subseteq Y$, it is birational. 

\end{remark}

\subsection{An intermediary stack}\label{ss:Kstar}

\begin{figure}
    \centering
    \begin{tikzcd}
            &\Kt \ar[dr, "\Gamma"] \ar[dl, "/", swap] \ar[dd, "/\tilde \Gamma"]\ar[dr, "v", swap]        \\
    \Mx \ar[dr, "/S_k", swap]\ar[dr, "\psi"]         &   &\Ks \ar[r] \ar[dl, "(d:1)^J \times /S_{\# I}"]\ar[dl, "\phi", swap]    &\Ms[g, R](X)       \\
            &\Kb
    \end{tikzcd}
    \caption{The stacks of stable maps to a fixed target $X$ used in this paper. 
    $\Ms[g, R](X)$ is the ordinary space of stable maps to $X$. 
    The rest are spaces of stacky genus-zero maps $\tilde C \to \Sym d X$. These can be interpreted as ramified finite covers $C'\to C$ of nonstacky curves together with a map $C'\to X$ satisfying a stability condition. The difference between $\Mx, \Kt, \Ks, \Kb$ lies in which points of $C', C$ are ordered. The maps between them are quotients by various groups reordering the marked points, except for $\Ks \to \Kb$, which is a quotient followed by a $d^{\# J}$-sheeted cover. }
    \label{fig:zoostacks}
\end{figure}

The moduli stack $\KpBSd$ parametrizes a triple $C \leftarrow C' \to D$ of curves over any base $S$. We introduce a variant $\KpBSdstar$ to describe how virtual classes push forward in Proposition \ref{prop:pfwdvfcs}.

The functor $p : \KpBSd \to \Mp$ sends such a triple to $D$. The map $p$ is proper of degree $e = k!(g!)^{\#J}(g!)^k$ \cite[\S 3]{mycostellogeneralization}. This map forgets everything about $C' \to C$, including the ordering of the forgotten marked points under $C' \to D$. 

Let $A \subseteq \num \ell$ be the points forgotten under $p$ as in Figure \ref{fig:excostellodiscretedata}. Let $\KpBSdstar$ be the space similar to $\KpBSd$, but where the marked points $A \subseteq \num{\ell}$ are \textit{unordered}. The forgetful map $\KpBSd \to \KpBSdstar$ is a torsor under
\[\Gamma := (S_g)^J \times S_g \wr S_k.\]
The group $S_g \wr S_k := S_k \ltimes (S_g)^k$ is the wreath product. There is a short exact sequence
\[1 \to S_g^k \to S_g \wr S_k \to S_k \to 1\]
and a section $S_k \dashrightarrow S_g \wr S_k$ of the quotient. We may view $S_g \wr S_k$ as a subgroup of $S_{gk}$ by choosing an identification of $\num{gk}$ with $\num{g} \times \num{k}$. 

The $k$ copies of $S_g$ reorder the unramified points in the fibers with simple ramification points, while $S_k$ reorders the fibers themselves and their images $\num k \subseteq C$.

\begin{example}

Isomorphisms of curves in $\KpBSdstar$ need not stabilize the unordered marked points. For example, $\PP^1$ with three unordered points has automorphism group $S_3$ by interchanging the points $0, 1, \infty$. The moduli space of genus zero curves with three unordered points is then $BS_3$.

\end{example}

These choices ordering certain marked points can also be made on the moduli of stable maps to the stack $\Sym d X$.

\begin{definition}

Let $\Kt \subseteq \tilde{\cal K}_{0, n}(\Sym d X)$ be the moduli space of representable stable maps to $\Sym d X$ with discrete invariants $\Xi$. This parameterizes \'etale, $d$-sheeted covers $\tilde C' \to \tilde C$ with minimal stack structure together with stable maps $\tilde C' \to X$. The curves may have nontrivial gerbes at marked points. All the marked points of $C'$ and $C$ are ordered. 

Define $\Ks$ analogously to $\KpBSdstar$ by forgetting the ordering on the marked points of $C'$ corresponding to $A \subseteq \num{\ell}$. See Figure \ref{fig:zoostacks}. 

\end{definition}

Extend \eqref{eqn:costellosquare} to the cartesian diagram
\begin{equation}\label{eqn:bigcostellosquare}
\begin{tikzcd}
\Kt \ar[rr, bend left=15, "q"]  \ar[r, "v"] \ar[d, "\pi'"] \pb        &\Ks \ar[r, "w"] \ar[d] \pb        &\Ms[g, R](X) \ar[d, "\pi"]       \\
\KpBSd \ar[r, "\nu"]  \ar[rr, bend right=15, "p", swap]     &\KpBSdstar \ar[r, "\omega"]         &\Mp[g, R].
\end{tikzcd}
\end{equation}

\begin{lemma}\label{lem:propbirational}
The map $\omega : \KpBSdstar \to \Mp[g, R]$ is proper and birational. 
\end{lemma}

\begin{proof}

Fix a generic smooth $D$ and prescribed ramification divisor $B = \sum_{i \in I} d(i) [i]$ over $\infty$ specified by $\Xi$. The proof of \cite[Theorem 3.12]{mycostellogeneralization} shows there is exactly one cover $C' \to C$ with ramification in $B$ and $C' \to D$ a partial stabilization. The map is thus proper and of pure degree one, but this is not yet sufficient by Remark \ref{rmk:puredegneqbirational}.

We argue that $\omega$ is generically representable, hence birational. The proof of \cite[Theorem 3.12]{mycostellogeneralization} shows that if $D \in \Mp[g, R](X)$ is general, the preimage under $\omega$ is exactly one cover $C'\to C$ with $C'= D$. Consider automorphisms
\[
\begin{tikzcd}
C' \ar[d] \ar[r, no head, "\sim"]      &C' \ar[d]     \\
C \ar[r, no head, "\sim"]       &C
\end{tikzcd}    
\]
of the map $C'\to C$. These form a subgroup of automorphisms of $C'$ because $C'\to C$ is an epimorphism. Since the map $\Aut(C'\to C) \to \Aut(D)$ is injective, the map is generically representable and hence birational.

\end{proof}

Hironaka's pushforward theorem \ref{thm:ourhironakathm} equates their fundamental classes: 
\[\omega_* \stquot{\OO_{\KpBSdstar}} = \stquot{\OO_{\Mp}} \quad \in \HH(\Mp).\]
Costello's pushforward theorem \ref{thm:ourcostellothm} likewise equates the virtual fundamental classes:

\begin{proposition}\label{prop:pfwdvfcs}

The fundamental class pushes forward along the map $w$ in Diagram \eqref{eqn:bigcostellosquare}:
\[\omega_* \stquot{\OO_{\KpBSdstar}} = \stquot{\OO_{\Mp}} \quad \in \HH(\Mp).\]

As a result, the virtual fundamental class pushes forward the same way:
\[w_* [\OO^{vir}_{\KsSymXstar}] = [\OO^{vir}_{\Ms[g, R](X)}] \quad \in \HH(\Ms[g, R](X)).\]

\end{proposition}

The wreath product $S_g \wr S_k$ arises naturally as the automorphism group of the projection $\num g \times \num k \to \num k$ of marked points of $C' \to C$:

\begin{remark}[Thanks to J. Rufus Lawrence]\label{rmk:iteratedsym}

The iterated stack-theoretic symmetric product $\Sym k {\Sym g X}$ is isomorphic to the global quotient
\[\Sym k {\Sym g X} \simeq X^{gk}/(S_g \wr S_k).\]
To make sense of $\Sym k -$ applied to a stack $\scr S$, use its functor of points
\[\Sym k {\scr S} (T) = \left\{
\begin{tikzcd}
T' \ar[d, "k:1", swap] \ar[r]      &\scr S         \\
T
\end{tikzcd}
\right\},\]
where $T' \to T$ is a $k$-sheeted cover that is part of the moduli. 

Consider the evaluation map corresponding to a point $C' \to C$ of $\KsSymXstar$ over a base $T$. Upon ordering the $k$-marked points $i$ of $C$, we get a map to $\Sym g X$ corresponding to each $i$. Ordering the $k$-marked points of $C$ entails a $k$-sheeted cover of $T$ with a map to $\Sym g X$. These evaluation maps are precisely
\[\KsSymXstar \longrightarrow \Sym k {\Sym g X} = X^{gk}/(S_g \wr S_k).\]

\end{remark}

\subsubsection{Two more stacks $\Mx, \Kb$}

The stack $\Kt$ is the simplest because all the marked points of $C'$ and $C$ are ordered, but we will not actually use it for our theorem. The variant $\Ks$ above is virtually birational to $\Ms[g, R](X)$. We need two more variants, completing Figure \ref{fig:zoostacks}. 

\begin{definition}

Let $\Mx$ be the moduli space of representable twisted stable maps $C \to \Sym d X$. It is the same as $\Kt$, except the marked points of $C'$ are not ordered. The only difference from twisted stable maps $C \to \Sym d X$ in the literature is the nontrivial gerbes. 

The quotient maps $\Kt \to \Mx, \Kt \to \Ks$ forget different marked points, so there is not a map between them. Define $\Kb$ to forget all the marked points of both, so only the points of $C$ that are not in $\num{k} \subseteq C$ are ordered. 

\end{definition}

As in Figure \ref{fig:zoostacks}, there are quotient maps from $\Kt$ to all the others $\Mx, \Kb, \Ks$ by various groups reordering marked points. The map $\Mx \to \Kb$ quotients by $S_k$, while $\Ks \to \Kb$ is a composite of many $d$-sheeted covers indexed by $J$ and a quotient by $S_{\#I}$.

\begin{remark}\label{rmk:phiisomonepoint}

Remark that $R = 1$ when $\# I = 1$ and $J = \varnothing$. In that case, $k = 3g$. 
The map $\KsSymXstar \to \bar{\cal K}$ is an isomorphism precisely when one of the conditions holds:
\begin{itemize}
    \item $R = 1$
    \item $\# I = 1$ and $g = 0$.
\end{itemize}
See Figure \ref{fig:zoostacks}.

\end{remark}

\begin{remark}\label{rmk:vfcspullback}
The finite étale maps 
\[
\begin{tikzcd}
\Kt \ar[r, "v"] \ar[d]         &\Ks \ar[d, "\phi"]        \\
\Mx \ar[r, "\psi"]         &\Kb
\end{tikzcd}
\]
in Figure \ref{fig:zoostacks} all equate virtual fundamental classes under pullback:
\[v^*[\OO_{\Ks}^{vir}] = [\OO_{\Kt}^{vir}], \qquad \text{etc.}\]
\end{remark}

\subsection{Gromov--Witten invariants in the $\HH$-theory of stacks}\label{ss:gwkthyforstacks}

Quantum $\HH$-theoretic invariants have been defined variously in the literature \cite{kthytoricstacks1} \cite[\S 2.4]{kthytoricstacks}. Our definition parallels \cite{abram-graber-vistoli}, adding in $\psi$ classes. Our invariants differ by a scaling factor due to conventions over whether gerbes at marked points are trivialized; see \cite{kthytoricstacks1}, \cite[Remark 2.8]{kthytoricstacks} or the original \cite[\S 4.4, 4.5]{AGVproceedingsfirstversion} for comparison. We allow nontrivial gerbes.

For any moduli space $\cal K$ of stable maps $C \to Y$ from $n$-pointed curves, there is an evaluation map
\[ev : \cal K \to Y^n.\]
If $Y$ is a stack and we take representable twisted stable maps $\tilde C \to Y$ in say $\cal K = \Ms(Y)$, the ``points'' of $\tilde C$ are not quite points, but $\mu_r$-banded gerbes. The evaluation map doesn't produce points of $Y$, but cyclic gerbes mapping representably to $Y$. Cyclic gerbes representably embedded in $Y$ form the \textit{rigidified cyclotomic inertia} stack $\instackbar{Y}$ \cite[\S 3]{abram-graber-vistoli}, so the evaluation map is
\[ev : \cal K \to (\instackbar{Y})^n.\]

The stack $\instackbar{Y}$ is closely related to the inertia stack $IY$. The universal gerbe over $\instackbar{Y}$ can be identified with representable maps from the trivial gerbe $\Hom^{rep}(B\mu_r, Y)$. Fixing an isomorphism $\mu_r \simeq \ZZ/r$ over $\CC$, we get a map to the inertia stack
\[
\begin{tikzcd}
\Hom^{rep}(B\mu_r, Y) \ar[r] \ar[d]       &IY         \\
\instackbar{Y}
\end{tikzcd}
\]
by composing $B\ZZ \to B\mu_r \to Y$. See Section \ref{ss:mapc} for a worked example.

Instead of pulling back $\HH$-theoretic classes from $Y$, we pull back from $\HH(\instackbar{Y})$. In our case, $Y = \Sym d X$ and our evaluation map is
\[ev_{C, \infty} : \KsSymX \to \instackbar{\Sym d X}.\]

%Let $\scr S$ be a Deligne--Mumford stack separated and of finite type over $\CC$. There is a \textit{rigidified cyclotomic inertia} stack $\instackbar{\scr S}$ parameterizing representable maps $\scr G \to \scr S$ from nontrivial gerbes $\scr G$ banded by finite cyclic groups $\mu_r$ \cite[\S 3]{abram-graber-vistoli}. The map $\scr G \to *$ gives a restriction map 
%\[\scr S \simeq \HHom(*, \scr S) \to \instackbar{\scr S}.\]
%One may view $\instackbar{\scr S}$ as a gerbey quotient of representable maps $\bigsqcup_r \Hom^{rep}(B\mu_r, \scr S)$ from trivial gerbes, but we do not need this. 

%Evaluation maps are obtained by restricting a stable map $C \to X$ to its marked points. A (representable) twisted stable map to a stack $\tilde C \to Y$ may have nontrivial gerbes $\scr G \subseteq \tilde C$ as marked points; restricting to them yields evaluation maps that land in $\instackbar{Y}$:
%\[ev : \cal K_{g, n}(Y) \to \instackbar{Y}.\]
%In our case, $Y = \Sym d X$ and our evaluation map is
%\[ev_{C, \infty} : \KsSymX \to \instackbar{\Sym d X}.\]

Now we add in $\psi$ classes. Write $\cal U \to \Mp$ for the universal curve and $L_i$ for the conormal bundle at the $i$th marked section $\sigma_i$:
\[L_i := N_{\sigma_i} \simeq \sigma_i^* T_{\cal U/\Mp}.\]
Use the same notation for their pullback to any moduli space with prestable curves, for example $\cal K_{g, n}(Y), \Ms$, etc. The classes of $L_i \in K^\circ(\Mp)$ and their pullbacks to various moduli stacks are referred to as $\psi$ classes.

Beware that maps between moduli spaces involving stabilization do not have the same $\psi$ classes. We compare the $\psi$ class of a map $C \to X$ with the stabilization $C^{st}$ in \S \ref{s:YCpsi}.

The stack $\cal K_{g, n}(Y)$ supports an obstruction theory relative to $\Mp$ that lets us define virtual fundamental classes $\OO^{vir}_{\cal K_{g, n}(Y)}$ in $\HH(\cal K_{g, n}(Y))$. Classes $\alpha_1, \dots, \alpha_n \in \HH(\instackbar{Y})$ and exponents $e_1, \dots, e_n$ give rise to a (descendent) \emph{Gromov--Witten invariant}
\[\num{\alpha_1 L_1^{e_1}, \dots, \alpha_n L_n^{e_n}} = \chi(\OO^{vir}_{\cal K_{g, n}(Y)} \otimes \prod ev_i^* \alpha_i \otimes L_i^{e_i}) \quad \in \QQ.\]
We are equally interested in power series of these invariants.

When $Y = \Sym d X$, there are two twisted curves $\tilde C' \to \tilde C$ and hence two evaluation maps and two sets of $\psi$ classes. The $\psi$ classes of $\tilde C'$ are the same as the marked points of $\tilde C$ below. The main technical problem in Theorem \ref{thm:genusgvgenuszero} will be converting between classes pulled back along the evaluation map of $\tilde C'$ and that of $\tilde C$. Genuine Gromov--Witten invariants have classes $ev^*\alpha$ pulled back from the evaluation map of $\tilde C$, not that of $\tilde C'$. This convention parrots \cite{abram-graber-vistoli}.

\subsection{Permutation-Equivariant $K$-theory}\label{ss:permeqvtkthy}

The $K$-theory of $\Kb$ is equivalent to permutation-equivariant $K$-theory of $\Mx$, as in \cite{giventalpermeqvtqkthy1}. Ordinary quantum $K$-theory entails ``correlators'' defined as the integrals
\[\num{\alpha_1 L_1^{m_1}, \dots, \alpha_n L_n^{m_n}} := \chi(\OO^{vir}_{\Kb} \otimes \prod ev_i^*(\alpha_i) L_i^{m_i}) \qquad \in K_\circ(\pt) = \QQ.\]

%\Leo{Maybe don't need full Givental story here, just equivariant $K$-theory. Maybe cut out a paragraph below, replace with more discussion of equivariant stuff?}
\begin{comment}
Enrich these correlators as $\Gamma$-representations by putting identical entries in the marked points corresponding to orbits of $\Gamma$ \cite{giventalpermeqvtqkthy1}. If we pick one $\alpha L^M$ for the $gk$ simply ramified fibers of $C' \to C$ and another $\beta_1 L^{N_1}, \dots, \beta_{\# J} L^{N_{\#J}}$ for the forgotten points in the fibers of $C' \to C$ over $J$, the correlator of 
\begin{itemize}
    \item[$J$:] $g$ copies of $\beta_1 L^{N_1}$, $g$ copies of $\beta_2 L^{N_2}$, etc. 
    \item[$k$:] $gk$ copies of $\alpha L^M$
    \item[$R$:] arbitrary insertions over the points not permuted by $\Gamma$
\end{itemize}
is a virtual $\Gamma$-module. The action of $\Gamma$ permutes the ordering of the factors corresponding to the points permuted by $\Gamma$. Each cohomology group in the alternating sum $\chi$ can be computed before or after pullback by elements of $\Gamma$, yielding the action. 
\end{comment}

How can we compute Euler characteristics of $\Kb$ by working on $\Mx$? Take the pullback square
\[
\begin{tikzcd}
\Mx \ar[r] \ar[d, "\psi", swap] \pb         &\pt \ar[d]       \\
\Kb \ar[r]         &BS_k \ar[r, "-^{S_k}"]        &\pt
\end{tikzcd}
\]
Denote pushforward along $\Kb \to BS_k$ by $\chi_{S_k}(-)$. This remembers the $S_k$-representation on the virtual vector space $\chi(\psi^*(-))$. To get the ordinary Euler characteristic $\chi(-)$ on $\Kb$, we have to take the virtual $S_k$-invariants of the virtual representation $\chi_{S_k}(-)$:
\[\chi(-) = \left(\chi_{S_k}(-)\right)^{S_k}.\]
These should be derived invariants, but we work over $\QQ$. The higher group cohomologies of $S_k$ valued in a representation all vanish, so the distinction is moot.

The class $\psi^*(-)$ on $\Mx$ will be a Gromov--Witten invariant where the insertions at the marked points forgotten under $\psi$ are identical. In our case, we only care about insertions away from those forgotten marked points. We will only integrate classes at ordered marked points of both $\Kb$ and $\Mx$. 

Even if we insert away from the permuted points $A \subseteq \num{\ell}$, the action of $\Gamma$ still nontrivially permutes the sections:

\begin{example}
Consider $S_3$ acting on $\overline{M}_{0,4} = \PP^1$ by permuting the last three marked points. The generating function of the $S_3$-invariant quantum $K$-invariants on $\overline{M}_{0,4}$ with only $L_1$ can be calculated
\[
 \chi \left( \left[ \overline{M}_{0,4} / S_3 \right], \frac{1}{1-q_1 L_1} \right) = \frac{1}{(1-q_1^2)(1-q_1^3)}.
\]
Indeed, $L_1 = \mathscr{O}(1)$, $\overline{M}_{0,4} = \PP^1$, and $H^{\geq 1} ( \overline{M}_{0,4}, L_1^d)=0$ for all $d \geq 0$. The case $d=1$ has the sections the linear functions on $\PP^1$, which are never $S_3$-invariant (up to M\"obius transformations). It is easy to see that $d=2$ and $d=3$ have invariant sections. In fact, $[ \overline{M}_{0,4} / S_3 ] = \PP(2,3)$ and the formula follows \cite{eulercharM11M04computation}.
\end{example}

\subsection{Grothendieck-Riemann-Roch}

This expository section explains why Grothendieck-Riemann-Roch (GRR) does not reduce equivariant Euler characteristics to ordinary ones on stacks, the way it would for schemes.

Using Grothendieck-Riemann-Roch for schemes, one would expect an equality
\[[\OO^{vir}_{\KsSymX}] \overset{?}{=} [\OO^{vir}_{\KsSymXstar}]^{\oplus \# \Gamma} \qquad \in \HH(\KsSymXstar)_\QQ.\]
Coupled with the projection formula, this would reduce permutation equivariant integrals to ordinary ones.

For schemes, this holds -- let $\pi : P \to X$ be a $G$-torsor with $P, X$ schemes for some finite group $G$. GRR gives a commutative square
\[\begin{tikzcd}
\HH(P) \ar[r] \ar[d, "\pi_*", swap]      &A_*(P)_\QQ \ar[d, "\pi_*"]         \\
\HH(X) \ar[r]      &A_*(X)_\QQ.
\end{tikzcd}\]
The Todd classes cancel out since $P \to X$ is \'etale and $T_X|_P = T_P$, so one can take the horizontal arrows as the Chern character isomorphisms. Since the pushforward in Chow groups gives $\pi_*[P] = \# G \cdot [X]$ and the horizontal isomorphisms send 1 to 1, we have
\begin{equation}\label{eqn:torsorpfwdfmlafails}
\pi_*[\OO_P] = [\OO_X^{\oplus \# G}]. 
\end{equation}
These formulas do not hold for stacks! 

Take $P = \pt$, $X = BG$. The analogous GRR square
\begin{equation}\label{eqn:GRRptBG}
\begin{tikzcd}
\HH(\pt) \ar[r] \ar[d]      &A_*(\pt)_\CC \ar[d]         \\
\HH(BG) \ar[r]         &A_*(BG)_\CC
\end{tikzcd}    
\end{equation}
still commutes \cite[\S 5]{edidinrrdmstacks}, and the Todd class terms vanish. But the lower horizontal arrow is not multiplicative and does not send 1 to 1! One cannot identify $\pi_* \OO_\pt$ and $\OO_{BG}^{\oplus \# G}$.

\begin{example}

Let $G = \ZZ/2$ and consider the quotient map $\pt \to B G$. The inertia stack is 
\[I B G = B G \sqcup \pt,\]
so its rational Chow groups are $A_*(IBG) = \CC^{\oplus 2}$. The GRR square for the quotient $\pi : \pt \to BG$ is then
\[\begin{tikzcd}
K_\circ(\pt) \ar[r, "ch(-)"] \ar[d, "\pi_*"]        &A_*(\pt) \ar[d, "\pi_*"]       \\
K_\circ (BG) \ar[r, "ch(-)"]        &A_*(IBG) \simeq \CC^{\oplus 2}.
\end{tikzcd}
\]
The Todd classes are trivial here. 

The Chern character of the trivial representation $\OO_{BG}$ is $(1, 1)$. By GRR \cite[Theorem 5.4]{edidinrrdmstacks}, the Chern character of $\pi_* \OO_\pt$ is 
\[ch(\pi_* \OO_\pt) = (2, 0) \qquad \in A_*(IBG) \simeq \CC^{\oplus 2}.\]
We can see that $\pi_* \OO_\pt \neq \OO_{BG}^2$. 

\end{example}

\begin{remark}

One can define a Borel equivariant $\HH$ theory for stacks in which formula \eqref{eqn:torsorpfwdfmlafails} holds using \cite{homotopyofstacksnoohi}. One can equip them with virtual fundamental classes and study Borel equivariant quantum $K$-theory. 

\end{remark}

The problem with \eqref{eqn:GRRptBG} is that $\pt \to BG$ introduces stack structure. It is representable, but points in $BG$ have more automorphisms than $\pt$ does. We want the opposite of representable, that the automorphism groups of points surject. We can prove a version of \eqref{eqn:torsorpfwdfmlafails} in this setting \cite{ourexpositoryquantumk}.

\section{Stabilization and $\psi$ classes}\label{s:YCpsi}

\subsection{Costello's lemma for stabilizing $\psi$ classes}
%\Leo{How much of this can we comment out?}
We recall the following three categories defined in \cite[\S 3]{costello}: $\Gamma^u$, which contains the label of nodal curves; $\Gamma^t$, which contains the label of twisted nodal curves; and $\Gamma^c$, which contains labels of twisted marked curves $\mathcal{C}$ and $\mathcal{C}'$ with an \'etale morphism $\mathcal{C}' \rightarrow \mathcal{C}$. They are related by the diagram:
\[
\begin{tikzcd}
\Gamma^c \arrow[r, shift left, "{s,t}"] \arrow[r, shift right] & \Gamma^t \arrow[r, "r"] & \Gamma^u,
\end{tikzcd}
\]
where $r,t$ stand for source and target of the \'etale map $\mathcal{C}' \rightarrow \mathcal{C}$ and $r$ maps $\mathcal{C}$ to its coarse moduli space. These categories will depend on a semigroup $A$. Furthermore, one can relate the functors between graphs into morphisms of stacks of moduli of curves by applying the functor $\mathfrak{M}$.

Given $\eta \in \Gamma^c$ consider the diagram
\[
\mathfrak{M}_{\eta} \xrightarrow{s} \mathfrak{M}_{s(\eta)} \xrightarrow{r} \mathfrak{M}_{r(s(\eta))} \xrightarrow{\pi} \mathfrak{M}_{\nu(I)},
\]
where $I\subsetneq T(s(\eta))$ is a chosen finite subset such that after removing tails in $I$, $s(\eta)$ remains stable. Let $\nu(I)$ be obtained from $r(s(\eta))$ by removing the tails in $I$ and denote the resulting contraction map by $\pi$.

To compare the pullback of $\psi$ classes via $s,r$ and $\pi$, we define the notation $S(e,t,I)$ as follows: Consider the diagram
\[
\mathfrak{M}_{\gamma} \rightarrow \mathfrak{M}_{r(s(\eta))} \xrightarrow{\pi} \mathfrak{M}_{\nu(I)},
\]
where $\gamma \rightarrow r(s(\eta))$ is a contraction of $\Gamma^u$. Now let $t \in T(\gamma) \setminus I$ and $e\in E(\gamma)$. We define
\[
S(e,t,I):= \begin{cases}
1, \mbox{ if $t$ is in a vertex of $\gamma_e$ that is contracted after forgetting the tails $I$; }
\\
0, \mbox{ otherwise}.
\end{cases}
\]
Here $\gamma_e$ is obtained by contracting all edges of $\gamma$ except $e$.

Costello described the pullback of $\psi$ classes in Chow:
\begin{lemma}[{\cite[\S 4.1]{costello}}]
For each $t \in T(\nu(I))$, we have
\[
s^* r^* \pi^* (\psi_t) = m(t) \psi_t - \sum_{f: \gamma \rightarrow \eta}    S(f,t,I)[\mathfrak{M}_f],
\]
where the sum is over $f: \gamma \rightarrow \eta$ in $\Gamma^c$ with $\# E(t(\gamma)) = \# E(f)=1$, $\mathfrak{M}_f \hookrightarrow \mathfrak{M}_{\eta}$ is the closed substack supported on the image of $f$, and $S(f,t,I) := \sum_{e\in E(s(\gamma))} m(e) S(e,t,I)$ is the corresponding multiplicity.
\end{lemma}

A similar result holds in $K$-theory:

\begin{lemma}
For each $t \in T(\nu(I))$, we have
\[
\begin{split}
s^* r^* \pi^* (L_t) &= L_t ^{\otimes m(t)} \otimes \OO\big(- \sum_{f: \gamma \rightarrow \eta} S(f,t,I)\mathfrak{M}_f \big) 
,
\end{split}
\]
with $f: \gamma \rightarrow \eta$ and $S(f,t,I)$ defined in previous lemma. 
\end{lemma}
We would like to write $\mathcal{O}(k\mathfrak{M}_f)$ in terms of a torsion sheaf. 
%Let $f: \gamma \rightarrow \eta$ in $\Gamma_c$ with $\# E(f) \geq 1$. We write $\mathfrak{M}_{f_e}$ to be the divisor corresponding to $e\in E(f)$. Note that $\OO_{\mathfrak{M}_f} = \prod_{e\in E(f)} \OO_{\mathfrak{M}_{f_e}}$.
%\begin{lemma}
%For each $t \in T(\nu(I))$, we have
%\[
%s^* r^* \pi^* (L_t) = L_t^{\otimes m(t)} + \sum_{f:\gamma \rightarrow \eta} (-1)^{\# E(f)} \prod_{e\in E(f)} \Big( 1+L_e+\dots + L_e^{S(f_e,t,I)-1}\Big) \OO_{\mathfrak{M}_f} ,
%\]
%where the sum is over $f:\gamma \rightarrow \eta$ in $\Gamma_c$ of type $0$ with $\# E(t(\gamma)) = \# E(f)$. 
%\end{lemma}
%\begin{proof}
%Let $D$ be a divisor with $r^* D = mD'$. Then we have
%\[
%\begin{split}
%    r^* \OO_D = r^* (\OO - \OO(-D) ) &= \OO - \OO(-mD') 
%    \\
%    &= \OO - (\OO-\OO_{D'})^m = (1+L + \dots + L^{m-1}) \OO_{D'},
%\end{split}
%\]
%where $L$ is characterized by $\OO_{D'} \cdot \OO_{D'} = (1-L)\OO_{D'}$. 

%The general (higher codimension) case follows easily from 
%\[
%\OO_{\mathfrak{M}_f} = \prod_{e\in E(f)} \OO_{\mathfrak{M}_{f_e}}.
%\]
%\end{proof}

\subsection{Explicit formulas for stabilizing $\psi$ classes in $K$-theory}

Before addressing stable maps to a target $V$, we work on the moduli of curves. Let $\pi: \Ms[g,m+n] \rightarrow \Ms[g,m]$. 
We introduce the following notations:

\begin{figure}
    \centering
    \begin{tikzpicture}[scale = .7]
            % A clipped circle is drawn
        \begin{scope}
            \clip (-1.5,-1.5) rectangle (1,1.5);
            \draw (0,0) ellipse (1.5 and 1);
        \end{scope}
        \begin{scope}
            \clip (2,-1.5) rectangle (4.5,1.5);
            \draw (3,0) ellipse (1.5 and 1);
        \end{scope}
        \node at (-1.5, 1.5){$C$};
        %connect the two pieces
        \draw (1, .7454) ..controls(1.5, .55).. (2, .7454);
        \draw (1, -.7454) ..controls(1.5, -.55).. (2, -.7454);
        %holes
        \draw (-.5, .08) ..controls(0, -.08).. (.5, .08);
        \draw (-.3, 0) ..controls(0, .08).. (.3, 0);
        \draw (2.5, .08) ..controls(3, -.08).. (3.5, .08);
        \draw (2.7, 0) ..controls(3, .08).. (3.3, 0);
        %rational tail on 1
        \fill[white] (-.75, -1) circle (.75);
        \fill[white] (-1.85, -2) circle (.75);
        \fill[white] (-3, -3) circle (.75);
        \draw (-.75, -1) circle (.75);
        \draw (-1.85, -2) circle (.75);
        \draw (-3, -3) circle (.75);
        %marked points on rational tail on 1
        \node[left] at (-.75, -1){14};
        \fill (-.9, -1.3) circle (.08);
        \node[right] at (-.75, -1){7};
        \fill (-.5, -1.3) circle (.08);
        \node[above] at (-.75, -1){5};
        \fill (-.95, -.5) circle (.08);
        \node[left] at (-1.85, -2){8};
        \fill (-1.85, -2) circle (.08);
        \node[above] at (-3.5, -3){4};
        \fill (-3.5, -3) circle (.08);
        {\color{violet}
        \node[left] at (-2.5, -3){1};
        \fill (-2.5, -3) circle (.08);}
        %rational tail on 2
        \fill[white] (1.5, 1.05) circle (.75);
        \fill[white] (2.2, 2.4) circle (.75);
        \fill[white] (2.9, 3.7) circle (.75);
        \draw (1.5, 1.05) circle (.75);
        \draw (2.2, 2.4) circle (.75);
        \draw (2.9, 3.7) circle (.75);
        %marked points on rational tail on 2
        \node[left] at (2.7, 3.7){6};
        \fill (2.7, 3.7) circle (.08);
        {\color{violet}
        \node[right] at (3, 3.7){2};
        \fill (3.1, 3.7) circle (.08);}
        \node[right] at (2.5, 2.4){10};
        \fill (2.6, 2.4) circle (.08);
        \node[left] at (2, 2.4){13};
        \fill (2, 2.4) circle (.08);
        \node[below] at (2.2, 1.9){12};
        \fill (2.2, 2) circle (.08);
        \node[left] at (1.5, 1.05){11};
        \fill (1.5, 1.05) circle (.08);
        %rational tail on 3
        \fill[white] (3.45, -.92) circle (.75);
        \draw (3.45, -.92) circle (.75);
        %marked points on rational tail on 3
        \node[below left] at (3.35, -.82){9};
        \fill (3.25, -.82) circle (.08);
        {\color{violet}
        \node[below right] at (3.55, -.82){3};
        \fill (3.65, -.82) circle (.08);}
        %Marked points
        {\color{violet}\fill (-.2, -.5) circle (.08);
        %\node[left] at (-3, -3){1};
        %\node[left] at (-.2, -.5){1};
        \fill (1.5, .3) circle (.08);
        %\node[right] at (1.5, .3){2};
        \fill (3, -.3) circle (.08);
        %\node[right] at (3, -.3){3};
        }
        \begin{scope}[shift={(10, 0)}]
            \begin{scope}
            \clip (-1.5,-1.5) rectangle (1,1.5);
            \draw (0,0) ellipse (1.5 and 1);
        \end{scope}
        \begin{scope}
            \clip (2,-1.5) rectangle (4.5,1.5);
            \draw (3,0) ellipse (1.5 and 1);
        \end{scope}
        \node at (-1.5, 1.5){$C$};
        %connect the two pieces
        \draw (1, .7454) ..controls(1.5, .55).. (2, .7454);
        \draw (1, -.7454) ..controls(1.5, -.55).. (2, -.7454);
        %holes
        \draw (-.5, .08) ..controls(0, -.08).. (.5, .08);
        \draw (-.3, 0) ..controls(0, .08).. (.3, 0);
        \draw (2.5, .08) ..controls(3, -.08).. (3.5, .08);
        \draw (2.7, 0) ..controls(3, .08).. (3.3, 0);
        {\color{violet}\fill (-.2, -.5) circle (.08);
        \node[left] at (-.2, -.5){1};
        \fill (1.5, .3) circle (.08);
        \node[right] at (1.5, .3){2};
        \fill (3, -.3) circle (.08);
        \node[right] at (3, -.3){3};}
        \end{scope}
    \end{tikzpicture}
    \caption{The rational tails corresponding to the decoration $\underline{a} = (\underline{a}_1, \underline{a}_2, \underline{a}_3)$ of type $0$, with\\
    $\underline{a}_1 = (a_{111}) (a_{121}) (a_{131}, a_{132}, a_{133}) = (4)(8)(5,7,14)$\\
    $\underline{a}_2 = (a_{211}) (a_{221}, a_{222}, a_{223})(a_{231}) = (6) (10, 12, 13) (11)$ \\
    $\underline{a}_3 = (a_{311}) = (9)$. \\
    The smallest numbered marked point on each rational tail (except 1,2, and 3) must be on the $\PP^1$ farthest from the main component to be type 0. 
    %The type $0$ condition means that on each rational tail, the numbers decrease as the components get farther from the main genus $g$ component. 
    }
    \label{fig:decorations}
\end{figure}

%\Leo{Remove some of the marked points -- superfluous}

\begin{itemize}
    \item (Decoration). Decorations index trees of rational curves to be contracted under forgetting points and stabilizing. See Figure \ref{fig:decorations}. 
    
    For the special case $m=1$, we denote a decoration of degree $r$ as follows:
    \[
    \underline{a} = (a_{1,1},\dots, a_{1,n_1})\dots (a_{r,1}, \dots, a_{r,n_r}).
    \]
    We further assume that 
    \[
    \{ a_{1,1}, \dots,a_{1,n_1}, \dots, a_{r,1},\dots,a_{r,n_r} \} \subset \{2,3,\dots,n+1\} = [2,n+1] 
    \]
    and
    \[
    a_{i,1}< a_{i,2} < \dots < a_{i,n_i}
    \] 
    for all $i$. 
    
    For the general case, a corresponding decoration is denoted as
    \[
    \underline{a} = (\underline{a}_1,\underline{a}_2,\dots,\underline{a}_m)
    \]
    where each $\underline{a}_i$ is a decoration in the special case $m=1$ and their disjoint set union forms a subset of $[m+1,m+n]$. We also denote it by $\underline{a}$ if no confusion may occur.
    \item (Degeneration strata). Given $m=1$ and a decoration $\underline{a}$, we define the corresponding degeneration strata of codimension $r$ on $\Ms[g,1+n]$ as follows:
    \[
    \Ms[g,1+n] \supset D_{1,\underline{a}} := \left ( \bigcap_{1\leq i \leq r} D_{1 a_{1,1} \dots a_{1,n_1}a_{2,1}\dots a_{i,n_i}} \right ),
    \]
    where $D_{abcde...}$ is the divisor with markings $abcde...$ lie on the rational tail. This is the closure of the locus where the curves have rational tails indexed by $\underline{a}$. 
    
    For the general case, given a decoration $\underline{a},$ we define the corresponding stratum 
    \[
    \Ms[g,m+n] \supset D_{\num{m},\underline{a}} := \bigcap_{1\leq i \leq m} D_{i,\underline{a}_i}.
    \]
    \item (Normal bundle). Given $m=1$ and a strata $D_{1,\underline{a}}$ as above, we have
    \[
    \OO_{D_{1,\underline{a}}}^2 = \lambda_{-1}(\oplus_{i=1}^r L_{1i}) \OO_{D_{1,\underline{a}}}.
    \]
    Here $L_{1i}$ is the normal bundle of the $i$th node.
    %\Leo{What does it have to do with smoothing it? Isn't it just the normal bundle of the $i$th node, period? }
    %\YC{You are right! I was thinking the global section of normal bundle is the parameter space of the deformation (if there is no obstruction). I usually  understand the deformation of nodal curve as smoothing the node.}
    More precisely, it can be described as follows:
    \[
    \OO_{D_{1,a_{1,1}\dots a_{i,n_i}}}^2 = \lambda_{-1}(\tilde{L_i})\OO_{D_{1,a_{1,1}\dots a_{i,n_i}}}
    \]
    Then we define
    \[
    L_{1i} :=  \tilde{L_i} |_{D_{1,\underline{a}}}.
    \]
    For the general case, we denote by $L_{ji}$ the normal bundle of the $i$th node of the tail containing the marking $j$ for $1\leq j\leq m$.
    
    \item (Type). Given $m=1$ and a decoration $\underline{a}$ or its corresponding strata $D_{A_1\underline{a}}$, we define its \emph{type} to be $0$ if 
    \[
    a_{1,1} <a_{j,k} \mbox{ for any $(j,k)\neq (1,1)$};
    \]
    we define its type to be $\underline{l} = (l_1\dots l_s)$ if
    \[
    \begin{split}
        &a_{1,1} > a_{2,1} > \dots > a_{l_1-1,1}
        \\
        &a_{l_1,1} > a_{l_1+1,1} > \dots > a_{l_2-1,1}
        \\
        &\qquad\qquad \vdots
        \\
        &a_{l_s,1}> a_{l_s+1,1} > \dots > a_{r,1}
    \end{split}
    \]
    and 
    \[
    a_{1,1} > a_{l_1,1} > \dots > a_{l_s,1}.
    \]

    Given a general decoration $\underline{a} = (\underline{a}_1,\dots, \underline{a}_m)$, we define its type on each $\underline{a}_i$ as in the special case and also denote it by $\underline{l}$ if no confusion may occur. We also say it is of type $0$ if $\underline{a}_i$ is of type $0$ for all $i$.

    \item (Polynomial corresponding to decoration). Given $m=1$ and a decoration $\underline{a}$ of type $(0)$, we define its corresponding polynomial with $r$ variables to be
    \[
    F_{1,\underline{a}}(x_1,\dots,x_r) =  (1-\prod_{i=1}^r x_i).
    \]
    For a partition $\underline{a}$ of type $(l_1,\dots,l_s)$, we define its corresponding polynomial to be
    \[
    F_{1,\underline{a}}(x_1,\dots, x_r) = \sum_{j=0}^s \Big (1- \prod_{i=l_j}^{l_{j+1}-1} x_i \Big).
    \]
    Here we set $l_0=1$ and $l_{s+1} = r+1$.
    
    For the general case, given a decoration $\underline{a} = (\underline{a}_1,\dots, \underline{a}_m)$, we define $F_{i,\underline{a}_i}$ as in the special case for all $i=1,\dots,m$.
    \item (Difference operator). 
    We define the difference operator $\delta$ on a (multi-)variable polynomial $F(\underline{x}) = F(x_1,\dots,x_n)$ as follows:
    \[
    \delta(F(\underline{x})) := \frac{F(\underline{x}) -\sum_i F(\underline{x})|_{x_i=1} + \dots + (-1)^n F(\underline{x})|_{x_1=\dots=x_n=1}}{(1-x_1)\cdots (1-x_n)} 
    \]
    %\Leo{in what sense is this induction? Base case?}
    %\YC{I change a way to define it. Is that clear now?}
\end{itemize}
\begin{lemma}
Let $\pi: \Ms[g,m+n] \rightarrow \Ms[g,m]$. Then we have 
\[
\pi^* (L_1) = L_1 + \sum_{\underline{a}:{\rm type}\ 0} (-1)^{{\rm codim} D_{\underline{a}}} \OO_{D_{\underline{a}}} \in K^0(\Ms[g,m+n]). 
\]
\end{lemma}
\begin{proof}
We decompose $\pi$ as 
\[
\Ms[g,m+n] \xrightarrow{\pi_{m+n}} \Ms[g,m+n-1] \rightarrow \dots \rightarrow \Ms[g,m+1] \xrightarrow{\pi_{m+1}} \Ms[g,m]
\]
and compute $\pi^* (L_1) = \pi_{m+n}^* \dots \pi_{m+1}^*(L_1)$ step by step. 

Given a strata $D \subset \Ms[g,m+n]$ of type $0$, we consider the inclusion-exclusion formula:
\[
\pi_{m+n}^* \OO_{\pi_{m+n}(D)} = \sum_{i=1}^r \OO_{D_i} - \sum_{i<j} \OO_{D_i\cap D_j} + \dots,
\]
where $\pi_{m+n}^*( \pi_{m+n}(D)) = \cup_{i=1}^r D_i$ with $D_i$ irreducible stratas.

Note that $\OO_D$ will show up in either the first or the second term of the right hand side depending on whether ${\rm ft}_{n+m}(D)$ is stable or not. Since $\pi_{n+m}(D)$ is still of type $0$, The lemma follows by induction.

\end{proof}

\begin{remark}
This expression is not symmetric with respect to indices $m+1,\dots ,m+n$ since we chose a special order of pull-backs. Different choices of pull-back order will result in different expressions. Nevertheless, any expression will give the same element in $K^0(\Ms[g,n+m])$.
\end{remark}

\begin{lemma} Let $\pi: \Ms[g,1+n] \rightarrow \Ms[g,1]$, $\underline{a}$ be a decoration, and $G$ be any polynomial or power series. Then
\[
{\rm Coeff} ( \OO_{D_{1,\underline{a}}}, \pi^*G(L_1)) = \delta ( G(L_1 - F_{1, \underline{a}})),
\]
where $F_{1,\underline{a}} = F_{1,\underline{a}}(L_{11}, \dots, L_{1, \deg \underline{a}})$ with $L_{1j}$ defined above. When applying $\delta$-operator on the right hand side, we view $G(L_1 - F_{1, \underline{a}})$ as polynomials with variables $\{L_{ij}\}$ and view $L_1$ as constant.
\end{lemma}
\begin{proof}
We start with a special case $m=1$ and $n=2$. By previous lemma, we have
\[
\pi^* L_1 = L_1 - \OO_{D_{12}} - \OO_{D_{13}} -\OO_{D_{123}} + \OO_{D_{(12)(3)}}.
\]
Let $G(x)$ be any polynomial. To compute ${\rm Coeff}(\OO_{D_{(12)(3)}}, G(\pi^* L_1))$, we introduce the following process:

Write $x=\OO_{D_{12}}$ and $y = \OO_{D_{123}}$ and hence $\OO_{D_{(12)(3)}} = xy$. Now we have
\[
\begin{split}
    &{\rm Coeff}(\OO_{D_{(12)(3)}}, G(\pi^* L_1))
    \\
    & \qquad = \frac{G(L_1 - x-y +xy)+ G(L_1 -y) + G(L_1-x) + G(L_1)}{xy} \Big |_{x=1-L_{11},y=1-L_{12}}.
    \\
    & \qquad = \delta \Big(G(L_1 - (1-L_{11}L_{12}))\Big).
\end{split}
\]
Some remarks are in order:
\begin{itemize}
    \item $L_{11}$ and $L_{12}$ are characterized by $x^2 = (1-L_{11})x$ and $y^2 = (1-L_{12})y$.
    \item The second equality follows from the definition of $\delta$. Here we view $G(L_1 - (1-L_{11}L_{12}))$ as polynomials in $L_{11}$ and $L_{12}$.
    \item $1-L_{11}L_{12}$ is exactly the polynomial  $F_{1,(12)(3)}(L_{11},L_{12})$, i.e. the polynomial corresponding to the strata $D_{(12)(3)}$. 
\end{itemize}
For general case, we can compute the coefficient of $D_{1,\underline{a}}$ using the above computation process. It suffices to find the polynomial corresponding to $D_{1,\underline{a}}$. To find the polynomial, we compute
\[
-\sum_{\substack{\underline{a}':{\rm type 0} \\ D_{1,\underline{a}'} \supset D_{1,\underline{a}} }} (-1)^{{\rm codim} D_{1,\underline{a}'} }
\OO_{D_{1,\underline{a'}}} \cdot \OO_{D_{1,\underline{a'}}} \Big| _{D_{1,\underline{a}}} = F_{1,\underline{a}}(L_{11}, \dots, L_{1 \deg \underline{a}}) \OO_{D_{1,\underline{a}}}.
\]
A direct computation shows that it is exactly the polynomial we defined above.
\end{proof}
\begin{theorem}
Let $\pi: \Ms[g,m+n] \rightarrow \Ms[g,m]$ and $G=G(x_1,\dots,x_m)$ be any polynomial of $m$ variables. Then
\[
\begin{split}
&\pi^* G(L_1,\dots, L_m) = G(L_1,\dots, L_m) 
\\
&\quad + \sum_{\underline{a}: {\rm all \ decorations}} \OO_{D_{\num{m},\underline{a}}} \delta \Big( G(L_1 -F_{1,\underline{a}_1}, \dots, L_m - F_{m,\underline{a}_m}) \Big),
\end{split}
\]
where
$
F_{i,\underline{a}_i} = F_{i,\underline{a}_i}(L_{i1}, \dots, L_{i \deg \underline{a}_i})
$
with $L_{ij}$ defined above. When applying the $\delta$-operator on the right hand side, we view 
\[
G(L_1 -F_{1,\underline{a}_1}, \dots, L_m - F_{m,\underline{a}_m})
\]
as polynomials with variables $\{L_{ij}\}$ and view $L_i$ as constant. 
\end{theorem}
\begin{proof}
It follows from the definition 
\[
D_{\num{m},\underline{a}} := \bigcap_{1\leq i \leq m} D_{i,\underline{a}_i},
\]
and the following observation that if $F(\underline{x}_1,\dots,\underline{x}_r)= \prod_{i=1}^r F_i(\underline{x}_i)$, then
\[
\delta \Big( F(\underline{x}_1,\dots,\underline{x}_r) \Big) = \prod_{i=1}^r \delta \Big( F_i(\underline{x}_i) \Big).  
\]
Here $\underline{x}_i$ could be multi-indices.
\end{proof}

\begin{corollary}
Let $\pi: \Ms[g,m+n] \rightarrow \Ms[g,m]$. Then we have
\[
\pi^* e^{\sum_{i=1}^m r_i L_i} = e^{\sum_{i=1}^{m} r_iL_i} + \sum_{\underline{a}: {\rm all \ decorations}} \OO_{D_{\num{m},\underline{a}}} \prod_{i=1}^m \delta \Big(e^{ r_i (L_i - F_{i,\underline{a}_i}) } \Big).
\]
Use the same assumptions as in the previous theorem when applying $\delta$-operator on the right hand side.
\end{corollary}
\begin{proof}
Take $G$ to be $e^{\sum_{i=1}^m r_iL_i}$ and apply the previous theorem. Notice that if $F(\underline{x}_1,\dots,\underline{x}_r)= \prod_{i=1}^r F_i(\underline{x}_i)$, then
\[
\delta \Big( F(\underline{x}_1,\dots,\underline{x}_r) \Big) = \prod_{i=1}^r \delta \Big( F_i(\underline{x}_i) \Big).  
\]
Here $\underline{x}_i$ could be multi-indices.
\end{proof}

\subsection{Application to the map $w: \Ks \to \Ms[g, R](X)$}

Continue to write $R = \ell-\# A = I \sqcup J$ for the number of marked points of $D$. Consider the map $p:\KpBSd  \rightarrow \Mp[g, R]$ sending a triple $C \leftarrow C' \rightarrow D$ to $D$ as before.

Write $M_i$, $L_i'$, and $L_i$ for the cotangent line bundles on $D$, $C'$ and $C$ respectively. Let $\iota : \num R \subseteq \num \ell$ be the inclusion of marked points -- $\iota(i) \in C'$ maps to $i \in D$. Write $\bar \iota (i)$ for the corresponding point of $C$.

Given $f: \gamma \rightarrow \eta$ lying over $D$, we define $F_{f,\bar \iota (i)}$ as follows:
\[
F_{f,\bar \iota (i)} = F_{i,\underline{a}_i} (L_{i1}^{m(1)}, \dots, L^{m(\deg \underline{a}_i)}_{i \deg \underline{a}_i}).
\]
The power $m$ is given by the ramification between nodes, $F_{i,\underline{a}_i}$ is defined in the point target case, and $\underline{a}_i$ is the combinatorial type of $f$. The definition of type is exactly the same as the point case.
\begin{proposition}\label{prop:psipb}
Let $G=G(x_1,\dots,x_R)$ be any power series
of $R$ variables. Then we have
\[
\begin{split}
& p^* \Big( G(M_1,\dots,M_R) \Big) = G(L_{\bar \iota (1)}^{m(t_1)},\dots,  L_{\bar \iota (R)}^{m(t_R)} )
\\
& + \sum_{f:\gamma\rightarrow \eta} \OO_{\frak{M}_f} \delta \Big( G(L_{\bar \iota (1)}^{m(t_R)} - F_{f,\bar \iota (1)}, \dots,  L_{\bar \iota (R)}^{m(t_R)} - F_{f,\bar \iota (R)} ) \Big),
\end{split}
\]
where $\delta$ only applies on variables $\{L_{ij}\}$.
\end{proposition}
\begin{proof}
The proof is similar to the point case. We only need to take care of the difference coming from ramification points.
%\Leo{these are not the stacky points -- they should happen at ramification points, not nodes. }

For the power on marked points, note that $p^* M_i = (L'_i)^{m(t_i)} +$ torsion part and $L'_i = L_{\bar \iota (i)}$ since $\tilde C' \rightarrow \tilde C$ is \'etale. 

For the definition of $F_{f,\bar \iota (i)}$, if $D$ is a divisor, note that
\[
\OO_{mD} := \OO - \OO(-mD) =  \OO - (\OO - \OO_D)^m = \delta (L_e^m) \OO_D,
\]
where $L_e$ is characterized by $\OO_D^2 = (1-L_e)\OO_D$. This explains the power in the definition of $F_{f,\bar \iota (i)}$.
\end{proof}

%\YC{Originally I wrote $m(t_i)L_{\bar \iota (i)}$ as the input of $G$. But now I think it should be $L_{\bar \iota (i)}^{m(t_i)}$. Let me know if I am wrong! }

\begin{remark}\label{rmk:psipb}

The cover $\tilde C' \to \tilde C$ is \'etale, so any marked point $i \in \tilde C'$ and its image $s \in \tilde C$ will have the same $\psi$ classes $L'_i = L_s$. Proposition \ref{prop:psipb} writes the pullback $p^* G(M_1, \dots, M_R)$ as a power series in $L'_i$; plugging in $L_s$ for each $L'_i$ in the fiber of $s \in \tilde C$ rewrites this pullback as a power series $H(L_1, \dots, L_n)$. This power series is invariant under $G$, giving an analogous formula on $\KsSymXstar$:
\[\omega^* G(M_1, \dots, M_R) = H^*(L_1, \dots, L_n).\]

\end{remark}

\begin{remark}\label{rmk:pfwdpowerseries}

The map $\phi : \Ks \to \Kb$ is a finite étale map. The $\psi$ classes of $\Ks$ are pulled back from those of $\Kb$. If $H(\vec L)$ is a power series with coefficients $\alpha_I \in K^{\circ}(\Ks)$ in $K$-theory, write 
\[H^\phi(\vec L)\]
for the power series on $\Kb$ with coefficients the pushforwards $\phi_* \alpha_I$ of the coefficients of $H(\vec L)$. If $\beta \in K_\circ(\Kb), \alpha \in K^{\circ}(\Ks)$, the projection formula equates
\[\phi_*(\phi^* \beta \otimes H(\vec L) \otimes \alpha) = \beta \otimes H^\phi(\vec L)\otimes \phi_* \alpha. \]

\end{remark}

\begin{remark}\label{rmk:spelloutHphi}

We sketch how to compute the coefficients $\OO_{\frak M_f}$ in terms of divisors on the moduli space of curves, and then how to deal with those divisors in quantum $K$-theory. This makes the power series $H^\phi(\vec L)$ computable.

A map of covers $f : \gamma \to \eta$ 
%\Leo{Is it a map of graphs or a map of covers}
%\YP{covers}
induces a map $f_* : \frak M_\gamma \to \frak M_\eta$ with image a closed substack $\frak M_f$. The map $f_* : \frak M_\gamma \to \frak M_f$ is a combination of a finite \'etale torsor for the automorphisms $\Aut(f | \eta)$ of the source over the target and a gerbe part for the stacky points. By Proposition \ref{prop:pfwdgerbestrsheaf}, the gerbe part does not affect the pushforward. The pushforward is then related to the regular representation of the sheaf $\Aut(f|\eta)$.

After the above reductions, it remains to explain that quantum $K$-invariants involving the torsion structure sheaf $\OO_S$ supported on boundary strata $S$ can be written in terms of the ``usual'' quantum $K$-invariants. In the cohomological Gromov--Witten theory, this is achieved by the splitting axiom, as the boundary stratum consist of the substacks indexed by the ``dual graphs'' exhibiting the imposed nodes of the general curves. In quantum $K$-theory a parallel splitting axiom is also available, albeit in a more sophisticated form. See \cite{Givental_2000_QK} and the genus reduction and splitting axioms in \cite[\S 4.3]{QK1} for details. One will then have to push these coefficients forward as outlined in \S \ref{ss:pfwdevalclasses}. 
%\Leo{Wonderful! Maybe more precise references for the two papers, but it looks good. }

\end{remark}

%%%%%%%%%%%%%%%%%%%%%%%%%%%%%%%%%%%%%%%%%%%%%%%%%%%%%%%%%%%%%%%%%%%%%%%%%%%%%%%%%%%%%%%%%%%%%%%%%%%%%%%%%%%%%%%%%%%%%%%%%%%%%%%%%%
\section{Main Theorem}

Let $\alpha_1, \dots, \alpha_R \in K^\circ(X)$ be classes and form
\[\alpha = \alpha_1 \boxtimes \cdots \boxtimes \alpha_R = \alpha_1|_{X^R} \otimes \cdots \otimes \alpha_R|_{X^R} \in K^\circ(X^R).\]
Let $G(\vec M)$ a power series in the $\psi$ classes of $\Ms[g, R](X)$ with coefficients arbitrary classes in $K^\circ(\Ms[g, R](X))$. 

Form the Gromov--Witten invariant
\[\chi(\OO^{vir}_{\Ms[g, R](X)} \otimes ev^* \prod \alpha_i G(\vec M)) \qquad \in K_\circ (\pt) = \QQ.\]
Write $H(\vec L)$ for the power series in the $\psi$-classes $L_i$ which is equal to $\omega^*G(\vec M)$ by Remark \ref{rmk:psipb}. Likewise write $H^\phi(\vec L)$ with the power series with coefficients given by the pushforwards of those of $H(\vec L)$ as in Remark \ref{rmk:pfwdpowerseries}.

\begin{theorem}\label{thm:genusgvgenuszero}

Gromov--Witten invariants on $\Ms[g, R](X)$ are equal to $S_k$-invariant Euler characteristics on the space $\Mx$ of stable genus zero maps to $\Sym d X$:
\[
\begin{split}
\chi(\OO^{vir}_{\Ms[g, R](X)} \otimes ev^* \alpha G(\vec M))      
&=\chi(\OO^{vir}_{\Ks} \otimes ev^* \alpha H(\vec L))         \\
&=\chi\left(\OO^{vir}_{\Kb} \otimes \phi_*(ev^* \alpha) H^\phi(\vec L)\right)      \\
&=\chi_{S_k}\left(\OO^{vir}_{\Mx} \otimes \psi^*\left(\phi_*(ev^* \alpha) H^\phi(\vec L)\right)\right)^{S_k}
\end{split}
\]

\end{theorem}

\begin{proof}

For the first equality, apply the projection formula for $w : \Ks \to \Ms[g, R](X)$ and the equality of virtual fundamental classes from Proposition \ref{prop:pfwdvfcs}. The evaluation maps are compatible and the power series $H(\vec L)$ is designed to be the pullback $w^*G(\vec M)$. 

The second equality results from the projection formula and the pullback
\[\phi^*[\OO^{vir}_{\Kb}] = [\OO^{vir}_{\Ks}].\]
The power series $H^\phi(\vec L)$ applies $\phi_*$ to the coefficients, so we are using the projection formula for each monomial of $H(\vec L)$ one at a time. 

The third equality results from the $S_k$-quotient $\psi : \Mx \to \Kb$ as in \S \ref{ss:permeqvtkthy}. The pullback
\[
\begin{tikzcd}
\Mx \ar[r] \ar[d, "\psi", swap] \pb         &\pt \ar[d]       \\
\Kb \ar[r]         &BS_k
\end{tikzcd}    
\]
equates the underlying vector space of $\chi_{S_k}(V)$ with the pullback $\chi(\psi^* V)$. The pushforward map $BS_k \to \pt$ then takes the quotient $(-)^{S_k}$ by $S_k$. 

\end{proof}

We have reduced Gromov--Witten invariants on $\Ms[g, R](X)$ to some equivariant Euler characteristics on $\Mx$. But are these Euler characteristics actually Gromov--Witten invariants?

\begin{lemma}\label{lem:pedanticevalmap}

The equivariant Euler characteristic $\chi\left(\OO^{vir}_{\Kb} \otimes \phi_*(ev^* \prod \alpha_i) H^\phi(\vec L)\right)$ is a ``Gromov--Witten invariant'' in genus zero. In other words, the class $\phi_* ev^* \prod \alpha_i$ can be described as the pullback of a class via the evaluation map of $\Kb$. 

\end{lemma}

We spend the rest of the section proving Lemma \ref{lem:pedanticevalmap}. This lemma was essentially left to the reader in \cite{costello}, although it is simpler in Chow groups than in $K$-theory. Reducing the part $H^\phi(\vec L)$ is left to the reader, following Remark \ref{rmk:spelloutHphi} and the process we outline for the evaluation classes.

\subsection{Turning equivariant Euler characteristics on $\Mx$ into proper Gromov--Witten invariants}\label{ss:pfwdevalclasses}

We need to show $\phi_* ev^* \alpha$ is pulled back from the evaluation map on $\Kb$. We first show it is pulled back from a natural map $\Kb \to (\Sym d X)^J \times \Sym{\#I}{X}$.

In the covers $C'\to C$ parameterized by $\Kb$, none of the marked points of $C'$ are ordered. Write $\cal Q_j \to \Kb$ for $j \in J$ for the $d$-sheeted cover of preimages of $j \in C$ in $C'$. Likewise, let $\cal P \to \Kb$ be the $S_{\#I}$-torsor ordering the preimages in $C'$ of $\infty \in C$. The product over $\Kb$ of all these $d$-sheeted covers and the $S_{\#I}$-torsor is $\cal K^*$:
\[\Ks = \prod_{\Kb} \cal Q_j \times_{\Kb} \cal P \to \Kb.\]

Considering the map $\cal Q_j \to X$ as a $d$-sheeted cover of $\bar{\cal K}$, it is parameterized by a map to $\Sym d X$
\[
\begin{tikzcd}
\cal Q_j \ar[r] \ar[d] \pb       &X \times \Sym{d-1}{X} \ar[d, "w_j"]  \\
\Kb \ar[r]   &\Sym d X.
\end{tikzcd}
\]
The equivariant map $\cal P \to X^I$ is also parameterized by a map to a symmetric product stack, but with different total space
\[
\begin{tikzcd}
\cal P \ar[r] \ar[d] \pb      &X^{\#I} \ar[d, "w_I"]   \\
\Kb \ar[r]    &\Sym {\#I} X.
\end{tikzcd}
\]

There is then a pullback square
\begin{equation}\label{eqn:pbevalmaps}
\begin{tikzcd}
\Ks \ar[r] \ar[d, "\phi"] \pb        &\prod \cal Q_j \times \cal P \ar[r] \ar[d] \pb       &X^J \times \left(\Sym {d-1} X\right)^J \times X^{\# I} \ar[d, "w"] \ar[r]       &X^R        \\
\Kb \ar[r, "\Delta"]    &\Kb^R \ar[r]     &\left(\Sym d X\right)^J \times \Sym {\#I} X. 
\end{tikzcd}
\end{equation}

Take classes $\alpha_1, \cdots, \alpha_R \in K^\circ(X)$. Write 
\[\alpha \coloneqq \alpha_1 \boxtimes \cdots \boxtimes \alpha_R\]
for the tensor product of the pullback of these classes to $X^J \times \left(\Sym {d-1} X\right)^J \times X^I$. Pullback and pushforward in $K^\circ$ theory along cartesian squares commute \eqref{eqn:projectionfmla}, so the resulting classes on $\Kb$ are the same
\begin{equation}\label{eqn:pfwdpbcompatible}
\phi_* (\alpha|_{\Ks}) = (w_* \alpha)|_{\Kb} \qquad \in K^\circ(\Kb).     
\end{equation}

\subsection{Evaluation maps}\label{ss:evalmaps}

We have shown $\phi_* ev^* \alpha$ is pulled back from $\Kb \to (\Sym d X)^J \times \Sym{\#I}{X}$. We need to compare this with the natural evaluation map on $\Kb$.

%\Leo{Sync with earlier mention of GWI's. }

Gromov--Witten invariants are certain integrals of $K$ theoretic classes pulled back from the evaluation maps as defined in Section \S \ref{ss:gwkthyforstacks}. To make sense of this, we need to be pedantic about the correct evaluation map for each target stack.

For spaces parameterizing multiple curves $C_1, C_2$ such as Hurwitz stacks, there is more than one evaluation map. We default to the evaluation maps of the base curve of the cover to be correct. 

\begin{example}

The space $\Mx$ parameterizes ramified covers $C'\to C$ together with a map $C'\to X$. There are $n$ ordered marked points of $C$ and $\ell$ unordered points of $C'$. The resulting evaluation map is
\[ev : \Mx \to \instackbar{\Sym d X}^n.\]

\end{example}

\begin{example}

The space $\Kb$ is similar to $\Mx$ but harder, because the $k$ gerbey points on $C$ are not even ordered. That part of the evaluation map lands in a symmetric stack of a symmetric stack
\[ev : \Kb \to \instackbar{
\left(\Sym d X\right)^J \times \Sym k {\Sym d X} \times \Sym {d} X
}.\]
The evaluation maps fit in a commutative square
\[
\begin{tikzcd}
\Mx \ar[r, "ev"] \ar[d]      &\instackbar{\left(\Sym d X\right)^J \times (\Sym d X)^k \times \Sym {d} X} \ar[d]   \\
\Kb   \ar[r, "ev"]      &\instackbar{\left(\Sym d X\right)^J \times \Sym k {\Sym d X} \times \Sym {d} X}.
\end{tikzcd}
\]

We won't need classes on the middle factor $\instackbar{\Sym k {\Sym d X}}$. Write $ev'$ for the projection away from this factor on the evaluation map of $\Kb$
\[ev': \Kb \to \instackbar{\left(\Sym d X\right)^J \times \Sym {d} X}\]

\end{example}

\begin{remark}

The 2-functor $\instackbar{-}$ does not distribute over products because of the representability requirement. For example, the identity map on $B \ZZ/2 \times B \ZZ/2$ is representable, but it doesn't factor through a representable map to either factor. It is more accurate to say the evaluation map lands in the product of $\instackbar{-}$ applied to each factor, so it is a product of the evaluation maps for each marked point. 

\end{remark}

Our ramification points are $\mu_r$-banded gerbes mapping to $BS_d$. Given a map $B\mu_r \to BS_d$, we can extract the set theoretic fiber of the stacky point of $\tilde C' \to \tilde C$ by taking the set-theoretic quotient $\num{d}/\mu_r$ of the corresponding action as in Remark \ref{rmk:mapsofgerbes}. 

More generally, we have a $\mu_r$-gerbe $\cal G \to T$ with a map $\cal G \to \Sym d X$. This means a finite étale cover $\tilde{\cal G} \to \cal G$ of degree $d$. The coarse moduli space of $\cal G$ is $T$, and that of $\tilde{\cal G}$ is a finite étale cover $T'\to T$. The degree $k$ of this cover is some number less than $d$. The map $\tilde {\cal G} \to X$ factors through $T'$ because $X$ is a scheme. This procedure gives a map
\begin{equation}\label{eqn:cmap}
c : \instackbar{\Sym d X} \to \bigsqcup_{k \leq d} \Sym k X,
\end{equation}
sending $\cal G \to \Sym d X$ to $T'\to X$. We describe this map in detail in the next Section \S \ref{ss:mapc}.

Because we fixed the discrete data $\Xi$ for $\Kb$, we know which component of each factor of $\bigsqcup_{k \leq d} \Sym k X$ it maps to
\[\begin{tikzcd}
\Kb \ar[d, "ev'", swap] \ar[r]         &(\Sym d X)^J \times \Sym{\# I}{X} \ar[d, "\inc"]      \\
\instackbar{\left(\Sym d X\right)^J \times \Sym {d} X}  \ar[r, "c"]        &\left(\bigsqcup_{k} \Sym k X\right)^{J \cup \infty}
\end{tikzcd}
\]
Let $\beta$ be a class in $K^\circ\left((\Sym d X)^J \times \Sym{\# I}{X}\right)$. Because the map $\inc$ is an inclusion of components, 
\[\inc^* \inc_* \beta = \beta.\]
Then the pullback $\beta|_{\Kb}$ is the same as the class
\[ev'^* c^* \inc_* \beta \qquad \in K^\circ(\Kb). \]

\begin{proof}[Proof of Lemma \ref{lem:pedanticevalmap}]

Take $\beta = w_* \alpha$ above. Then
\[\phi_*(\alpha|_{\Ks}) = (w_* \alpha)|_{\Kb} = {ev'}^*c^* \inc_* w_* \alpha,\]
using \eqref{eqn:pfwdpbcompatible} and the discussion immediately above. This expresses the factor $\phi_*(\alpha|_{\Ks})$ in the Gromov--Witten invariant as a class pulled back via the evaluation map on $\Kb$.

\end{proof}

\subsection{Describing the map $c$}\label{ss:mapc}

We describe the map $c$ in detail using the inertia stack $I(\Sym d X)$. This section is purely expository, and an example is given at the end.

The inertia stack of a quotient stack $\bra{Y/G}$ is the disjoint union of the quotients
\[I (\bra{Y/G}) = \bigsqcup_{g \in G \text{ conj classes}} \bra{Y^g/C_g}\]
of the fixed locus $Y^g$ by the centralizer $C_g \subseteq G$.

For $\Sym d X$, $Y = X^d$ and $G = S_d$. Conjugacy classes of $S_d$ are indexed by cycle types, the multiset of lengths of cycles. For example,
\[g_1 \coloneqq (12)(34)(5) \in S_5 \mapsto \{2, 2, 1\}\]
\[g_2 \coloneqq (234)(761)(5)(89) \in S_9 \mapsto \{3, 3, 1, 2\}.\]

Let $N_s$ be the number of cycles of length $s$, $N \coloneqq \sum N_s$ the total number of cycles, and $t$ the cardinality of the set of distinct lengths in the cycle type. 
We include all cycles of length one $N_1$, so $\sum i N_i = d$. 
For $g_1, g_2$ above,
\[g_1 \mapsto N_1 = 1, N_2 = 2, t = 2\]
\[g_2 \mapsto N_1 = 1, N_2 = 1, N_3 = 2, t = 3.\]
The centralizer of a cycle type $g$ is
\[C_g = S_{N_1}  \times S_{N_2} \rtimes (\ZZ/2)^{N_2}   \times S_{N_3} \rtimes (\ZZ/3)^{N_3} \cdots  \times S_{N_t} \rtimes (\ZZ/t)^{N_t} . \]
%by \cite{3588666}. 
The fixed locus $(X^d)^g \subseteq X^d$ is the multidiagonal
\[(X^d)^g = \{(x_1, \cdots, x_d) \, | \, x_i = x_j \text{ if }i, j \text{ in the same cycle}\},\]
so $(X^d)^g \cong X^N$. For $g_1, g_2$, we have 
\[(X^d)^{g_1} = \{(x, x, y, y, z)\} \in X^5\]
\[(X^d)^{g_2} = \{(x, x, x, y, y, y, z, w, w)\} \subseteq X^9.\]

Write $H_i = S_{N_i} \rtimes (\ZZ/i)^{N_i} = \ZZ/i \wr S_{N_i}$. There is an exact sequence
\[0 \to (\ZZ/i)^{N_i} \to H_i \to S_{N_i} \to 1\]
and a splitting $S_{N_i} \subseteq H_i$. The subgroups $\ZZ/i$ act trivially on the diagonal fixed locus $X \subseteq X^{N_i}$, so the stack quotient is a trivial gerbe
\[\bra{X/(\ZZ/i)} = X \times B \ZZ/i.\]
The quotient $X^N/\prod H_i$ is the product of symmetric products
\[X^N/\prod H_i = \prod_i \Sym {N_i} {\bra{X/(\ZZ/i)}} = \prod_i \Sym {N_i} {(X \times B\ZZ/i)}.\]

On each component, there is a map to a single symmetric product
\[
\prod \Sym {N_i} {(X \times B\ZZ/i)} \to \prod \Sym {N_i} X \overset{c'}{\longrightarrow} \Sym N X.
\]
The second map $c'$ takes $t$ covers $T'_i \to T$ of degrees $N_i$ and assembles them into one cover $T'= \bigsqcup T'_i \to T$ of degree $N$. This gives a map
\begin{equation}\label{eqn:rigidifiedcbar}
\bar c : I(\Sym d X) \to \bigsqcup_{N \leq d} \Sym N X.    
\end{equation}
The reader can check this coincides with the map $c$ defined in \eqref{eqn:cmap}. 

\begin{lemma}

There is a commutative diagram
\[
\begin{tikzcd}
I(\Sym d X) \ar[dr, "\bar c", bend left=15] \ar[d]               \\
\instackbar{\Sym d X}, \ar[r, "c", swap]      &\bigsqcup_{k \leq d} \Sym k X
\end{tikzcd}
\]
where $c$ is the map \eqref{eqn:cmap} and $\bar c$ is \eqref{eqn:rigidifiedcbar}. 
    
\end{lemma}

We explain the above from the point of view of covers. A $T$-point of $I(\Sym d X)$ is a map $T \times B \mu_r \to \Sym d X$, which is a $d$-sheeted cover $P \to T \times B \mu_r$. Write $P_0$ for the pullback $d$-sheeted cover of $T$
\[
\begin{tikzcd}
P_0  \ar[r] \ar[d] \pb       &P \ar[d]      \\
T \ar[r]       &T \times B \mu_r,
\end{tikzcd}
\]
so $P = \bra{P_0/\mu_r}$. 

Fix a generator $\ZZ/r \simeq \mu_r$. If $T$ is a geometric point, $P_0 \simeq \num{d}$ and $\mu_r \action P_0$ is an element $\sigma \in S_d$. Its order is the ${\rm lcm}$ of the cycle lengths, which divides $r$; $r$ is equal to the ${\rm lcm}$ when the map $T \times B \mu_r \to \Sym d X$ is representable. Reordering $P_0 \simeq \num{d}$ conjugates $\sigma$, so $\sigma$ is well defined as a conjugacy class. 

Even if $T$ is not a geometric point, this defines a locally constant function
\[T \mapsto S_d/_{ad} S_d; \qquad t \mapsto [\sigma]\]
from $T$ to the conjugacy classes of $S_d$. This decomposes $I(\Sym d X)$ into components corresponding to cycle type, or partitions of $d$.

Given a partition $d = \sum i N_i$ that is the cycle type of $\sigma \in S_d$, a point $T \to I(\Sym d X)$ factors through the corresponding component if its fibers at geometric points are isomorphic to $\num{d}/\sigma$. The $N_i$ different orbits of $i$ points may be interchanged in families over $T$, and each such orbit may vary in a $B\ZZ/i$ family. The collection of orbits of $i$ points is parameterized by the stack
\[\Sym {N_i} {(B\ZZ/i)}.\]
We have $\prod \Sym {N_i} {(X \times B\ZZ/i)}$ instead because we also need a map to $X$.

\subsubsection{The case $d = 5$}

If $d = 5$, the seven cycle types/partitions are 
\[5, 4 + 1, 3 + 2, 3 + 1 + 1, 2 + 2 + 1, 2 + 1 + 1 + 1, 1 + 1 + 1 + 1 + 1.\]

The corresponding components of $I(\Sym 5 X)$ are
\[X \times B\ZZ/5, \quad X^2 \times B\ZZ/4, \quad  X^2 \times B\ZZ/3 \times B\ZZ/2, \quad 
X/(\ZZ/3) \times X^2/S_2,
\]
\[ X \times X^2/(S_2 \rtimes (\ZZ/2)^2),  \quad X/\ZZ/2 \times X^3/S_3,  \quad X^5/S_5.\]
Project away from the cyclic gerbes $B \ZZ/i$
\[X,  \quad X^2,  \quad X^2, \quad  X \times \Sym 2 X,  \]
\[\Sym 2 X \times X,  \quad X \times \Sym 3 X,  \quad \Sym 5 X.\]

%\YP{Why $B\mu_4$? Why not $B\mu_2$ or something else?}
For example, consider a trivial gerbe $b : B\mu_4 \to \Sym 5 X$ mapping to the symmetric product. This corresponds to a 5-sheeted cover $P \to B \mu_4$, counting stacky multiplicity. 

The composite $\pt \to B \mu_4 \to \Sym 5 X \to BS_5$ parameterizes a 5-sheeted cover of the point, which we trivialize and view as $\num{5}$. The action of $\mu_4$ on $\num{5}$ can be viewed as a map $\mu_4 \to S_5$. Choose a generator to identify $\mu_4 \simeq \ZZ/4$ and let $\sigma \in S_5$ be the image of the generator. 

The element $\sigma$ has order dividing 4. Take for example
\[\sigma = g_1 = (12)(34)(5).\]
Its order $2$ is not $4$, so the classifying map $b$ is not representable.

The stack quotient of $\num{5}$ by $\ZZ/4 \cdot \sigma$ is a disjoint union
\[B\mu_2 \sqcup B \mu_2 \sqcup B \mu_4.\]
This is the 5-sheeted cover $P \to B\mu_4$. The corresponding point $\pt \to I(\Sym 5 X)$ factors through the component of the partition $2 + 2 + 1$, i.e. $X \times X^2/(S_2 \rtimes (\ZZ/2)^2)$ above.

The cover comes with a map $P \to X$. The map $\bar c$ takes coarse moduli spaces of $P \to B \mu_4$, obtaining $\num{3} \to \pt$. The map $P \to X$ factors through $\num{3}$ because $X$ is a scheme with no stack structure. This sends the component $X \times X^2/(S_2 \rtimes (\ZZ/2)^2)$ to $\Sym 3 X$.

\section{Elliptic curves example}

\begin{figure}
    \centering
\begin{tikzpicture}
\node at (-1, .75){$C'$};
\draw[->] (-1, .45) to (-1, -1.25);
\node at (-1, -1.5){$C$};
\draw[-] (0.75, 0) to (1.5, 0);
\draw[-] (0.75, 1) to (1.5, 1);
\draw (0.75, 1) arc (90:270:.5);
\draw (1.5, 0) arc (-90:90:.5);
\draw[-] (3.5, 0) to (5, 0);
\draw[-] (3.5, 1) to (5, 1);
\draw (3.5, 1) arc (90:270:.5);
\draw (6, 0) arc (-90:90:.5);
\draw[-] (0, -1.5) to (5, -1.5);
\node at (5.5, -1.5){$\cdots$};
\draw[-] (6, -1.5) to (7, -1.5);
\fill (.25, .5) circle (.07cm);
\fill (2, .5) circle (.07cm);
\fill (3, .5) circle (.07cm);
\fill[white] (6.5, .5) circle (.07cm);
\draw (6.5, .5) circle (.07cm);
\node at (5.5, 0){$\cdots$};
\node at (5.5, 1){$\cdots$};
%
%
%nodes on base for J
\fill (0.25, -1.5) circle (.07cm);
\fill (2, -1.5) circle (.07cm);
%nodes on source for J
%
%nodes on base for k
\fill (3, -1.5) circle (.07cm);
\node[below] at (2, -2){$k$};
%node on base for \infty
\fill[white] (6.5, -1.5) circle (.07cm);
\draw (6.5, -1.5) circle (.07cm);
\node[below] at (6.5, -2){$\infty$};
\end{tikzpicture}
    \caption{An elliptic curve $C'$ and its double cover of $C = \PP^1$ simply ramified at four points: $\Xi$ with $g = 1$, $d = 2$, $k = 0$. Marked points are colored white if remembered and black if forgotten under the map $p : \KpBSd \to \Mp[1, 1]$. }

\end{figure}

We apply our theorem in the case of elliptic curves, with $g = R = 1$, $d=2$, and $X = \pt$. The symmetric product is the classifying stack $\Sym d X = BS_d$. The results are reassuring but not surprising.

We know the map $w$ is proper and birational. For elliptic curves, more is true. 

\begin{lemma}

The map $w : \KsSymptstar \to \Ms[1, 1]$ is an isomorphism.

\end{lemma}

\begin{figure}
    \centering
    \begin{tikzpicture}
    \draw (.2,0) .. controls (-2,1) .. (.2,2);
    \draw (-.2,0) .. controls (2,1) .. (-.2,2);
    \node at (-2, 1){$C'_1$};
    \node at (2, 1){$C'_0$};
    \fill[white] (.65, .375) circle (.07cm);
    \draw (.65, .375) circle (.07cm);
    \fill (1.45, 1) circle (.07cm);
    \fill (-1.45, 1) circle (.07cm);
    \fill (-.65, .375) circle (.07cm);
    \draw[->] (0, -.5) to (0, -1.5);
    \draw[->] (2.5, 1) to (5.5, 1);
    \begin{scope}[shift={(0, -2)}]
    \draw (.2,.2) .. controls (-1,-.375) .. (-2,-.5);
    \draw (-.2,.2) .. controls (1,-.375) .. (2,-.5);
    \fill[white] (.65, -.18) circle (.07cm);
    \draw (.65, -.18) circle (.07cm);
    \fill (1.45, -.44) circle (.07cm);
    \fill (-1.45, -.44) circle (.07cm);
    \fill (-.65, -.18) circle (.07cm);
    \end{scope}
    \begin{scope}[shift={(7, 1)}, scale=1.2]
    \draw[variable = \x, domain=-1.3:1.3] plot (-\x*\x+1, \x*\x*\x-\x); 
    \fill[white] (1, 0) circle (.07cm);
    \draw (1, 0) circle (.07cm);
    \node at (1.5, 0){$D$};
    \end{scope}
    \end{tikzpicture}
    \caption{The admissible cover $C'\to C$ which stabilizes to the nodal cubic ${C'}^{st} = D$. The components of $C'$ are labeled $C'_0, C'_1$ and they each doubly cover a component of $C$. The pair of nodes $P$ is the intersection $C'_0 \cap C'_1$ of the components}
    \label{fig:singularcubicadmcover}
\end{figure}

\begin{proof}
%\YP{sorry, but I don't understand the very long proof.}
%\Leo{This is just what we did when you visited in Leiden. I wrote up what we talked about. }
The map on coarse moduli spaces is an isomorphism $\PP^1 \longequals \PP^1$. It remains to show the stack structure is the same; i.e., the automorphism groups of the admissible covers $C'\to C$ are the same as that of the stabilization $D = {C'}^{st}$. We already checked this for a generic elliptic curve in the proof of Lemma \ref{lem:propbirational}. 

We want to show $\Aut(C'\to C) = \Aut(D)$. Again, it helps that automorphisms of the covering map are a subgroup of automorphisms of the source $\Aut(C'\to C) \subseteq \Aut(C')$. These automorphisms must send ramification points to ramification points, and it will be clear that they also send the preimage $I \in C'$ of infinity to itself.

For all the smooth elliptic curves $D$, the source is already stable $C' \longequals D$ and the unique map $C' \to C = \PP^1$ is the quotient by the elliptic involution. For all smooth elliptic curves with automorphism group $\ZZ/2$, this identifies $\Aut(C'\to C) \simeq \Aut(D)$. The smooth curves $j = 0, 1728$ remain, as does the singular cubic. 

The curve $j = 1728$ has equation
\[y^2 = x^3 - x.\]
The automorphism group $\ZZ/4$ is generated by $(x, y) \mapsto (-x, iy)$. This commutes with the automorphism $x \mapsto -x$ of $\PP^1$. Likewise $j = 0$ has equation
\[y^2 = x^3 - 1\]
Letting $\zeta$ be a 6th root of unity, the automorphism group is generated by $(x, y) \mapsto (\zeta^2 x, \zeta^3 y)$. This also commutes with an automorphism of $\PP^1$.

For the singular elliptic curve $D$, the preimage is an admissible cover $C' \to C$ with both source and target reducible. See Figure \ref{fig:singularcubicadmcover}. Each consists of two components, labeled $C'_i, C_i$ for $i = 0, 1$. Assume the point in $I \subseteq C'$ lies on $C'_0$. 

The map restricts to two double covers $C'_i \to C_i$ of $\PP^1$'s. The preimage $P \subseteq C'$ of the singular point of $C$ is a pair of nodes joining $C'_0, C'_1$. 

Any automorphism of $C'$ must restrict to an automorphism of each component, because of the marked point. Any automorphism $\varphi$ of $C'$ restricts to an automorphism of the pair of nodes $P$ which determines $\varphi$. This is because $\varphi$ must preserve at least three points on each component, the nodes and the ramification points. All such automorphisms lie over $C$, so $\Aut(C'\to C) = \Aut(P) = \ZZ/2$. This is the same automorphism group of the singular stabilization $D = {C'}^{st}$, the nodal cubic.

\end{proof}

The map $\phi : \Kspt \to \Kbpt$ is also an isomorphism by Remark \ref{rmk:phiisomonepoint}, so Theorem \ref{thm:genusgvgenuszero} merely says that quantum $K$ invariants on $\Ms[1, 1] = \PP(4, 6)$ are equivariant Euler characteristics on its natural $S_3$-cover by $\Mxpt$.

This cover is pulled back from the quotient $\PP^1 \to \PP(2, 3)$ by $S_3$ acting on $\lambda \in \PP^1$. This map classically parameterizes the Legendre form
\[y^2 = x(x-1)(x - \lambda)\]
of an elliptic curve. The map $\Mxpt \to \Kbpt = \Ms[1, 1]$ then fits in a pullback square
\[
\begin{tikzcd}
\Mxpt \ar[r] \ar[d] \pb      &\Ms[1, 1] \ar[d]       \\
\PP^1 \ar[r, "/S_3"]       &\PP(2, 3),
\end{tikzcd}
\]
with vertical arrows $\mu_2$-gerbes. 
%\Leo{Where $-_c$ is the component where all the marked points are twisted.}
%\YP{With your permission, I'd prefer the notation to be $\Ms[0, 4] (B\ZZ/2, tw^4)$. In fact, I am hoping that our notation can be more suggestive. I'd prefer to change all $\Mx, \Kb$ etc..}
%\Leo{This would be quite a violent change. I think it's simpler to mention that $-_c$ is in the data of $\Xi$, so it's fine.}

\begin{comment}
    
\subsection{Branch Morphism}

There is a finite map $br : \Kb \to \Ms[0, n]$, a version of the Hurwitz branch map. It is a finite stacky map in general. For $n=4$ as above, it is the trivial $\mu_2$-gerbe 
\Leo{NOT a trivial $\mu_2$-gerbe! Make sure this argument makes sense; might have to delete everything. Remember the target of the branch should be \emph{unordered} marked points on $\PP^1$  with nontrivial gerbes at marked points?}
\[br : \PP(4, 6) \to \PP(2, 3).\]
We have $br_*[\OO^{vir}_{\Kb}] = [\OO^{vir}_{\Ms[0, 4]}]$ in this special case by Example \ref{ex:pfwdvfctrivialgerbes}. 
\Leo{Is this a trivial gerbe???? If not, nevermind!}

The $\psi$ classes of $\Kb$ are not pulled back from $\Ms[0, 4]$, so we cannot use the projection formula immediately. This is because $\Kb$ has a stacky curve $\tilde C$, while $\Ms[0, 4]$ only has its coarse moduli space $C$. The $\psi$ classes $\bar L_i$ of $C$ are squares of the $\psi$ classes $L_i$ of $\tilde C$ because it is simply ramified at each point.

\end{comment}

\bibliographystyle{alpha}
\bibliography{zbib}

\end{document}